\documentclass[11pt,letterpaper, reqno]{amsart}
\usepackage{amsmath,amsfonts,amsthm,amssymb,amscd,mathtools}
\usepackage{mathrsfs, enumerate}
\usepackage{latexsym}
\usepackage{tikz-cd}
\usetikzlibrary{decorations.pathreplacing}

\definecolor{cadmiumgreen}{rgb}{0.0, 0.42, 0.24}
\usepackage[
colorlinks, citecolor=cadmiumgreen,
pdfauthor={Jeffrey C. Lagarias, David Harry Richman}, 
pdfstartview ={FitV},
]{hyperref}
\hypersetup{
pdftitle={The floor quotient partial order},
}

\usepackage{times}
\usepackage[T1]{fontenc}

\usepackage{chngcntr}

\newtheorem{thm}{Theorem}[section]

\newtheorem{cor}[thm]{Corollary}
\newtheorem{lem}[thm]{Lemma}
\newtheorem{prop}[thm]{Proposition}

\theoremstyle{definition}
\newtheorem{defi}[thm]{Definition}
\newtheorem{exa}[thm]{Example}

\theoremstyle{remark}
\newtheorem{rmk}[thm]{Remark}

\newcommand{\RR}{\ensuremath{\mathbb{R}}}

\newcommand{\ZZ}{\ensuremath{\mathbb{Z}}}
\newcommand{\QQ}{\mathbb{Q}}

\newcommand{\NNplus}{{\mathbb{N}^+}}

\newcommand{\sA}{{\mathcal{A}}}

\newcommand{\sC}{\mathcal{C}}
\newcommand{\sD}{{\mathcal D}}

\newcommand{\sI}{{\mathcal I}}
\newcommand{\sK}{{\mathcal K}}
\newcommand{\sM}{{\mathcal M}}

\newcommand{\sP}{{\mathcal P}}
\newcommand{\sPdual}{{\mathcal P}^\ast}
\newcommand{\sQ}{{\mathcal Q}}

\newcommand{\sS}{{\mathcal S}}

\newcommand{\dd}{d}

\newcommand{\floor}[1]{\left\lfloor{#1}\right\rfloor}

\newcommand{\floorfrac}[2]{\floor{\frac{#1}{#2}}}

\newcommand{\angles}[1]{\langle {#1} \rangle}

\newcommand{\J}{{\Phi}} 
\newcommand{\AD}{{\,\preccurlyeq_{1}\,}}
\newcommand{\ADa}{{\,\preccurlyeq_{a}\,}}

\newcommand{\notAD}{{\,\not\preccurlyeq_{1}\,}}
\newcommand{\muAD}{{\mu_{1}}}
\newcommand{\muADa}{{\mu_{a}}}

\newcommand{\ADset}{\sQ}

\newcommand{\ADsetsmall}{\sQ^-}
\newcommand{\ADsetlarge}{\sQ^+}
\newcommand{\Dset}{\sD}
\newcommand{\ADmult}{\sM}
\newcommand{\cutset}[2]{\sK({#1},{#2})}

\newcommand{\divides}{\,\mid\,}

\newcommand{\Z}{{Z}}

\newcommand{\TC}{\mathrm{TC}} 
 
\numberwithin{equation}{section}

\begin{document}

\title{The floor quotient partial order}
\author{Jeffrey C. Lagarias}
\thanks{The research of the first author was supported by NSF grant
DMS-1701576}
\address{Dept.\ of Mathematics, University of Michigan, Ann Arbor MI 48109-1043} 
\email{lagarias@umich.edu}

\author{David Harry Richman}
\address{Dept.\ of Mathematics, University of Washington,
Seattle, WA  98195-4350}
\email{hrichman@uw.edu}


\date{October 22 2023}

\begin{abstract}
A positive integer $d$ is a floor quotient of $n$ if there is a positive integer $k$ such that ${d = \floor{{n}/{k}}}$.
The floor quotient relation defines a partial order on the positive integers.
This paper studies the internal structure of this partial order and its M\"{o}bius function.
\end{abstract} 
\maketitle

\setcounter{tocdepth}{1}
\tableofcontents

\bibliographystyle{amsplain}

%
%
\section{Introduction}

Denote the positive integers $\NNplus = \{ 1, 2, 3, \ldots\}$.  
We call a partial order 
$\sP = (\NNplus, \preccurlyeq)$
on the positive integers  an
{\em approximate divisor order}  if:
\begin{enumerate}
\item
It refines the (multiplicative) divisor partial order $\sD =(\NNplus, \divides)$, i.e.
\[
\dd \divides n \qquad\Rightarrow\qquad \dd \preccurlyeq n,
\]
in which $\dd \divides n$  denotes $\dd$ divides $n$.
\item
It is refined by the additive total order $\sA = ( \NNplus, \le)$, i.e.
$$
\dd \preccurlyeq n \qquad\Rightarrow\qquad \dd \leq n .
$$
\end{enumerate}
Any such order is intermediate between the multiplicative order and the additive order. 

This paper studies an approximate divisor order, 
defined using  the floor function. 
The floor function $\floor{x}$ maps  a real number $x$ to the greatest integer no larger than $x$.

\begin{defi}
A positive integer $\dd$ is  a {\em floor quotient} of $n$, 
written $\dd \AD n$,  
if $\dd = \floorfrac{n}{k} $ for some positive  integer $k$.
\end{defi} 

We show that the binary relation $\AD$ on $\NNplus$ is an approximate divisor order, and study its structural properties,
with a particular emphasis on the behavior of its M\"{o}bius function.

%
%
\subsection{Floor quotient partial order}\label{subsec:10} 


The floor quotient relation $\dd \AD n$ asserts that an interval of length $n$ can be
 cut into $k$ equal pieces of length $\dd$ with 
a possible  extra piece, of length shorter than $k$, 
which is discarded. If the extra piece is absent, then $\dd$ divides $n$.
In computer science, the  operation $n \mapsto \floorfrac{n}{k}$ is called integer division or  floor division (by $k$) and the resulting value  $\dd$  is called a  floor quotient, see Koren \cite{Koren:02}.



\begin{thm}\label{thm:approx-order} 
 The floor quotient relation $ \AD $
 defines a partial order  $\ADset \coloneqq ( \NNplus, \AD)$.
 The partial order $\ADset$ is an approximate divisor order.
\end{thm} 

This result is proved as Theorem \ref{thm:approx-order2}, as  a consequence of a result    
showing the equivalence of six properties characterizing floor quotients (Theorem~\ref{thm:equiv-properties}), see Section~\ref{subsec:31}.  
The first two of these are: 
\begin{enumerate}
\item (Cutting property)  
 $\dd = \floor{ \frac{n}{k}}$ for some integer $k$. 

\item (Covering property)  
For the half-open unit interval $I_\dd= [\dd, \dd+1)$, there is some integer
$k$ such that the interval
$k I_\dd$ covers $I_n$,
i.e.
$ k  [\dd, \dd+1) \supseteq [n,n+1)$.
\end{enumerate}

The partial order axioms for $\sQ = (\NNplus, \AD\!)$ follow 
 easily from the covering property (2).
The cutting property (1) implies  that the {floor quotient} order refines the divisor partial order, 
and is refined by the additive total order.

The floor quotient order $\AD$ strictly refines the divisor order.
Figure~\ref{fig:11} gives the Hasse diagram of the floor quotient relation on all floor quotients of $16$,
which has extra vertices and extra edges compared with the Hasse diagram for divisors of $16$.
\begin{figure}[h]
\centering 
 \begin{tikzpicture}[xscale=1.2, yscale=0.8]
  \node (16) at (0,6) {$16$};
  \node (8) at (-1,5) {$8$};
  \node (5) at (1,4.5) {$5$};
  \node (4) at (-2,3) {$4$};
  \node (3) at (1,1.5) {$3$};
  \node (2) at (-1,1) {$2$};
  \node (1) at (0,0) {$1$};
  \draw (16) -- (8) -- (4) -- (2) -- (1);
  \draw (16) -- (5) -- (2);
  \draw[out=-10, in=60] (16) to (3);
  \draw (3) -- (1);
  \end{tikzpicture}
\caption{Hasse diagram for the interval $ \ADset[1,16] = \{ d : d \AD 16\}$.} 
\label{fig:11}  
\end{figure}
 
We introduce some terminology.  
If  $\dd = \floor{n/k}$ is a floor quotient of $n$, then 
we will call $n$ a {\em floor  multiple} of $d$; 
we will also call any associated value $k$ a {\em cutting length} of $(\dd,n)$.
There may be more than one cutting length  $k$ yielding a given floor quotient pair  $(\dd, n)$.

The first  object of this paper is to systematically describe the internal structure of the floor quotient poset.
 The  second object is to study the M\"{o}bius function of this poset.  It has a relation to the
 M\"{o}bius function of the divisor order, which is described by the classical number-theoretic M\"{o}bius function.

%
%
\subsection{Structural properties of floor quotient  poset}\label{subsec:12a} 
 
Structural aspects of the floor quotient poset concern 
the size and symmetries of  order intervals,
the behavior of the set of all floor multiples of an integer, which forms a numerical semigroup,
and a symmetry exchanging cutting length numbers and gap numbers between floor quotients in an initial interval. 
These aspects are studied in Section \ref{sec:structure}. 
 
The initial intervals $\ADset[1,n] = \{d : 1 \AD d \AD n\}$ of the floor quotient order have a structure in which 
both  additive structures and multiplicative structures are visible.  
In 2010 Cardinal~\cite[Proposition 4]{Cardinal:10} noted a self-duality property of the set  $\ADset[1,n]$ under an involution.
Cardinal showed that:
\begin{enumerate}[(1)]
\item
 The floor quotient initial interval
 $\ADset[1,n]$ is the union $\ADsetsmall(n) \bigcup \ADsetlarge(n)$,
where
\[
	\ADsetsmall(n) = \{1,2,3,\ldots, s \}
	\qquad\text{and}\qquad
	\ADsetlarge(n) 
	= \{ n,  \floor{\frac{n}{2} }, \floor{\frac{n}{3} }, \ldots, \floor{ \frac{n}{s} } \},
\]
letting $s = \floor{\sqrt{n}}$.
\item
One has
\[
	\ADsetsmall(n) \bigcap \ADsetlarge(n) = \begin{cases}
	\{s\} &\text{if } \,  s^2\le n < s(s+1), \\
	\emptyset &\text{if } \, s(s+1) \le  n < (s+1)^2.
	\end{cases}
\]
\item
The map $\J_n(k) = \floor{n/k}$ 
acts as an involution on the set $\ADset[1,n]$, which exchanges the sets
$\ADsetsmall(n)$ and $\ADsetlarge (n)$ bijectively. 
\end{enumerate}
Cardinal's result is proved here, for the sake of completeness, as Proposition \ref{prop:interval-complement}.

 The floor quotient order on the set $\ADsetsmall(n)$ is, by definition,  the floor quotient order restricted
to the {\em additive} interval  $\{1, \ldots, \floor{\sqrt{n}}\}$.
The following result, proved  as Theorem ~\ref{thm:upper-interval-struct}, 
shows that the  floor quotient order restricted to $\ADsetlarge(n)$ records the {\em multiplicative} structure of $\{1, \ldots, \floor{\sqrt{n}}\}$.

\begin{thm}
\label{thm:intro-upper-interval-struct}
For a positive integer $n$, let
$ \ADsetlarge(n) = \{ \floor{ n/k } \,:\, 1 \le k \le \floor{ \sqrt{n}} \} $.
For $1 \le j, \ell \le \floor{\sqrt{n}}$ we have
\begin{equation} 
\Big\lfloor{\frac{n}{j}}\Big\rfloor \AD \floor{\frac{n}{\ell}}
\qquad\text{if and only if}\qquad
\ell \divides j.
\end{equation}
That is: the map $k \mapsto \floor{n/k}$ 
defines an anti-isomorphism of posets from 
the set $\{1,2,\ldots, \floor{\sqrt{n}}\}$ equipped with the divisor order
to $\ADsetlarge(n)$  with the floor quotient order.
\end{thm}

In Section~\ref{sec:2} we compare other aspects of the floor quotient order with the additive and multiplicative order.
We note two features. 
\begin{enumerate}
\item[(1)]
 The involution $\J_n$ acting on $\ADset[1,n]$
is order-reversing for the additive order.  
Consequently  the sets  $\ADsetsmall(n)$
and $\ADsetlarge(n)$ are reverse-order isomorphic for the additive
partial order.
The map  $\J_n$ is not order-reversing for the floor quotient order
and symmetry between the  floor quotient order structures on $\ADsetsmall(n)$
and $\ADsetlarge(n)$ is  broken. 

\item[(2)]
The multiplicative  divisor order on $\NNplus$ has a {scaling symmetry}, 
in the sense that the divisor intervals $\sD[d, n]$ and $\sD[ad, an]$ 
are order-isomorphic
under the linear map $x \to ax$, for any integer $a\ge 1$. 
The floor quotient order breaks
this  scaling symmetry; we quantify deviations from it in Section \ref{sec:interval-size}.
However, the family of floor quotient intervals $\ADset[ad, an]$ as $a$ varies does stabilize for sufficiently large $a$, thus having a ``limit scaling symmetry.''
\end{enumerate}

%
%
\subsection{M\"{o}bius function of the floor quotient partial order}
\label{subsec:12} 
The classical M\"{o}bius function $\mu$ on $\NNplus$ is given by
 \[
\mu(n) = \begin{cases} 
(-1)^k  &\text{if $n = p_1 p_2 \cdots p_k$ is squarefree, where $p_i$ are prime}, \\
0 &\text{if $n$ is not squarefree}.
\end{cases}
\]
It satisfies the fundamental identity
\begin{equation}
\label{eq:classical-mu-sum}
\sum_{d \divides n} \mu(d) = 0 \qquad\text{if } n \geq 2.
\end{equation}

For any locally finite partial order one can define a two variable M\"{o}bius function. 
Restricting to $\mathcal P = (\NNplus, \preccurlyeq)$ defined on the positive integers, 
the two variable M\"{o}bius function $\mu_{\sP} : \NNplus \times \NNplus \to \ZZ$ 
is defined by the relations (i) $\mu_{\sP}(d, d) = 1$ for all $d$, (ii) $\mu_{\sP}(d,n) = 0$ if $d \not\preccurlyeq n$, and (iii) $\sum_{d \preceq {e} \preceq n} \mu_{\sP}(d, e) = 0$ for all $d, n$.
For the divisor partial order
  $\sD = (\NNplus, \divides )$,
the   associated two-variable  M\"{o}bius function $\mu_{\sD}$ satisfies
\[
	\mu_{\sD}(m, n) = 
	\begin{cases}
	\mu(\frac{n}{m}) &\text{if }m \divides n, \\
	0 &\text{otherwise}.
	\end{cases}
\]

We let $\muAD$ denote the two-variable  M\"{o}bius function of the floor quotient order $\ADset= (\NNplus, \AD)$.
While $|\mu_\sD(m,n)| \leq 1$ for the divisor poset $\sD$, 
the  floor quotient M\"{o}bius values $\muAD(m, n)$  
do not appear to be bounded in magnitude, see Figures~\ref{fig:mobius-10e4} and \ref{fig:mobius-10e6}.

On part of its domain, namely for all 
$(d, n)$ with $d \geq \sqrt{n}$, 
the floor quotient M\"{o}bius function has a precise connection to the divisor M\"{o}bius function,
and on this domain $|\muAD(d, n)| \leq 1$ holds.
The following theorem is a consequence of Theorem \ref{thm:intro-upper-interval-struct}.

\begin{thm}
\label{thm:upper-interval-mobius2} 
Let $\muAD$ denote the two-variable M\"{o}bius function of the floor quotient poset.

\begin{enumerate}[(a)]
\item For any integer $k$ with $1 \leq  k   \leq {\sqrt{n}}$,
the function $\muAD$ satisfies
\begin{equation}
\label{eq:upper-mobius} 
\muAD \left(\floorfrac{n}{k} , {n} \right) = \mu(k)
 \end{equation}
 where $\mu(\cdot)$ is the classical M\"{o}bius function.  

\item For positive integers $d,n$ with $d \geq \sqrt{n}$,  
we have 
$|\muAD(d, n)| \leq 1.$

\end{enumerate}
\end{thm}

This theorem combines results given in Theorem \ref{thm:upper-interval-mobius} and Corollary~\ref{cor:63}.
The  results in Theorem \ref{thm:upper-interval-mobius2} cover the range $d \geq \sqrt{n}$,
which comprises the ``multiplicative'' part of the floor quotient poset.
When $d \leq \sqrt{n}$, 
the floor quotient M\"{o}bius function  $\muAD(d,n)$
behaves quite differently than the divisor M\"{o}bius function $\mu_\sD$.

The next result shows that the asymptotic growth of $\muAD(d,n)$ is at most polynomial as a function of the ratio ${n}/{d}$.
We note that the stated bound holds for all $d$ and $n$, and does not involve limits.
\begin{thm}
\label{thm:intro-mobius-bound} 
Let $\muAD: \NNplus \times \NNplus \to \ZZ$ 
denote the M\"{o}bius function of the  floor quotient order. 
Let $\alpha_0 \approx 1.729$ denote the unique positive real given by $\zeta(\alpha_0)=2$,
where $\zeta(s)$ is the Riemann zeta function.
Then the following upper bound holds:
\begin{equation}
\label{eq:intro-mobius-bound}
  |\muAD(d,n)| \leq  \left(\frac{n}{d}\right)^{\alpha_0} 
  \qquad \text{for all } d,n\in \NNplus .
\end{equation}
\end{thm}

\noindent 
This result  is proved
as Theorem \ref{thm:intro-mobius-bound2} in Section \ref{subsec:mobius-bounds}.
We prove the bound \eqref{eq:intro-mobius-bound} by first bounding the total number of chains
in the interval $\ADset[1,n]$, given in Section \ref{subsec:count-chains}.
Experimental data, given in Section \ref{subsec:mobius-data}, suggests that the asymptotic growth of $|\muAD(1,n)|$ is at least
polynomial, 
and that $ |\muAD(1,n)| > n^{0.6}$ for infinitely many positive integers $n$.

We remark that the constant $\alpha_0 \approx 1.729$ 
appeared earlier in the work of Kalm\'{a}r~\cite{Kalmar:1931} and Klazar--Luca~\cite{KlazarL:07},
who studied the asymptotic growth of the number of chains in divisor order intervals $\sD[1,n]$.
Chains in $\sD[1,n]$ are in bijection with ``ordered factorizations'' of $n$, and with ``perfect partitions'' of $n$.

The next result concerns sign changes in the floor quotient M\"{o}bius values $\muAD(1,n)$.
%
%
\begin{thm}
\label{thm:intro-mobius-sign-change0} 
There exists an infinite sequence of integers 
$\ell_1=2< \ell_2< \ell_3 < \cdots$ 
such that for each $j \ge 1$ the floor quotient M\"{o}bius function has
\begin{equation}
(-1)^j  \muAD(1, \ell_j) > 0, 
\end{equation}
with $\ell_j$ satisfying the bounds 
\begin{equation}
\label{eqn:sgn-bound} 
\ell_{j+1} \le 2\ell_j^2-2. 
\end{equation} 
Consequently the 
function $\muAD(n)= \muAD(1,n)$ has  infinitely many sign changes as $n$ varies.
\end{thm} 
This result is proved as Theorem \ref{thm:mobius-sign-change} in Section \ref{subsec:64}.
To prove it
we use a recursion derived from a cancellation between  $\muAD(1, 2n)$ and $\muAD(1,n)$,
 see Proposition \ref{prop:mu-recursion0}.
The bound \eqref{eqn:sgn-bound} only guarantees a sign change over a distance which grows quadratically. 
Empirical  data indicates 
$\muAD(1,n)$  
does keep constant sign for very long intervals, see the numerical data in Section \ref{subsec:mobius-data}. 
This behavior is in sharp contrast with the sign changes of the classical M\"{o}bius function $\mu(n)$. 

A new feature of the floor quotient M\"{o}bius function is
the existence of   a second  recursion for computing some M\"{o}bius function values,
which arises from its relation
to the additive total order.   
It concerns  the  differenced M\"{o}bius function 
\[
  \Delta\muAD(1,n) \coloneqq \muAD(1,n)-\muAD(1, n-1).
\]
The differenced values $\Delta\muAD(1,n)$ satisfy a recursion of a novel type, formulated below,
which feeds back integrated values $\muAD(1,s)$ at special values $n$ of form $n=s^2$ and $n = s(s+1)$.
\begin{thm}
\label{thm:intro-mobius-recursion}
For $n\geq 3$, the differenced M\"{o}bius function satisfies the recursion 
\begin{equation} 
\label{eq:intro-diff-mu-sum} 
 \sum_{\substack{ d \divides n \\  d > \sqrt{n} }} \Delta \muAD(1,d) = \begin{cases}
\displaystyle
0 & \text{if }n \ne s^2\,\, \mbox{or}\,\, s(s+1),\\
\displaystyle
  - \muAD(1,s) & \text{if }n = s^2 \text{ or } s(s+1) .
 \end{cases}
\end{equation}
Note that $s = \floor{\sqrt{n}}$ when $n=s^2$ or $s(s+1)$.
\end{thm}
This result is proved as Theorem \ref{thm:mobius-recursion} in Section \ref{subsec:diff-mobius}.
Using the recursion \eqref{eq:intro-diff-mu-sum} we show the  function $\Delta\muAD(1,n)$ takes many small values, and vanishes on 
a set of natural density $\approx 0.405$ (Proposition~\ref{th:65}). 
This recursion provides a potential approach to prove upper bounds on the growth rate of $\muAD(1,n)$. 
The analysis becomes  complicated, 
and we present only partial results.

There are many simple to state, unresolved problems about the behavior of the M\"{o}bius function $\muAD$,
see Section \ref{sec:concluding}. 

%
%
\subsection{Prior work}\label{subsec:14} 

This work was motivated  by 
work of  J.-P. Cardinal \cite{Cardinal:10}.
The main topic of Cardinal's paper is
the construction of a commutative algebra $\mathbf{\sA} = \mathbf{\sA}_n$ of integer matrices, 
 for each $n \ge 1$,  
whose  rows and columns  are indexed by the
floor quotients of $n$, which form  the initial interval $\ADset[1,n]$ of the floor quotient partial order. 
Cardinal  showed that a certain matrix  $\sM = \sM_n$ of the algebra $\sA_n$
 has explicit connections with the classical Mertens function  $M(n) = \sum_{j=1}^n \mu(j)$, 
 and he related asymptotic behavior of the norms $||\sM_n||$ of these matrices to the Riemann hypothesis \cite[Theorem 24]{Cardinal:10}. 
 The algebra ${\sA}_n$ can be identified with a subalgebra of the (noncommutative)
  incidence algebra $\sI(\sQ[1,n])$ of the initial interval 
 $\sQ[1,n]$ of the
 floor quotient partial order,  see Section \ref{subsec:82}.

 Cardinal proved  commutativity of the algebra $\mathbf{\sA}_n$ using 
a floor function identity which states the commutativity of 
certain dilated floor functions acting on the real line  
\cite[Lemma 6]{Cardinal:10}:
for positive integers $k, \ell $ one has 
the identities
\begin{equation}
\label{eq:floor-dilation-commute}
\floor{ \frac{1}{k} \floorfrac{n}{\ell} } 
= \floor{ \frac{1}{\ell} \floorfrac{n}{k} } 
= \floorfrac{n}{k\ell} \qquad\text{for all $n$}.
\end{equation}
These identities are the basis of the existence of the  floor quotient partial order;
 they imply  the transitivity property of the floor quotient relation. 
In \cite{LMR:16}, the authors together with T. Murayama classified  pairs of dilated floor functions $(\floor{\alpha x} , \floor{\beta x})$
that commute under composition as functions on $\RR$.
They showed that the identities \eqref{eq:floor-dilation-commute} are the only nontrivial cases
where the commutativity property of dilated floor functions occurs.
(The trivial cases are
(i) $\alpha=\beta$, 
(ii) $\alpha=0$,
and (iii) $\beta=0$.)

In the course of his analysis, 
Cardinal proved various essential properties of floor quotients,
in particular observing that $d \mapsto \bar{d} = \floor{n/d} $ 
acts as an involution on the set of floor quotients of $n$
\cite[Proposition 4]{Cardinal:10},  
given here as  Proposition \ref{prop:interval-complement}.
This involution lifts to a symmetry property of Cardinal's algebra. 
For  floor quotient intervals $\ADset[1,n]$ this symmetry is broken at the level of the partial order relation, i.e. the involution is generally not order-reversing for $\AD$, 
see Section \ref{subsec:24}.

There is much other literature related to floor quotients, and floor function identities. 
They appear in the counting estimates in  Dirichlet's hyperbola method~\cite{Dirichlet:1849}, see Figure~\ref{fig:integer-hyperbola} in Section~\ref{subsec:gaps-cutting-mults}.   
The result of Heyman~\cite{Heyman:19} treats from the hyperbola method  perspective the  cardinality
of the floor function interval $\ADset[1,n]$, already treated in \cite[Proposition 4, Remark 5]{Cardinal:10}.
Further convolution identities are studied by Cardinal and Overholt~\cite{CardinalO:20}.

Further  study of the floor quotient partial order  and its M\"{o}bius function may  shed light on  
both the divisor order and the algebra structure studied by Cardinal.

%
%
\subsection{Contents of paper}\label{subsec:13} 

 Section \ref{sec:2} presents an overview of structural properties of the floor quotient order,  
compared and contrasted with the divisor order on $\NNplus$ and the additive order on $\NNplus$,
for various statistics. 
It states various theorems of this paper without proof, viewed in this context.

The main part of the paper is Sections~\ref{sec:structure} through \ref{sec:mobius}.
 The  first objective of the paper  is to establish 
basic structural properties of the  floor quotient partial order, 
relating the quantities $n$, $d$ and $k$. 
This occupies Sections~\ref{sec:structure} through \ref{sec:interval-size}.
The second objective of the paper   is to study  the M\"{o}bius function 
$\muAD: \NNplus \times \NNplus \to \ZZ$ of the floor quotient poset.
This topic is treated in Section~\ref{sec:mobius}.
 Section~\ref{sec:concluding} raises four directions for further investigation.
 In more detail: 

\begin{itemize}
\item
Section~\ref{sec:structure} studies structural relations between the parameters
$(n, d, k)$ 
in the equation $d = \floor{n / k}$. 
Section \ref{subsec:31} gives six equivalent characterizations of the floor quotient relation in Theorem \ref{thm:equiv-properties}, 
and uses them to prove Theorem~\ref{thm:approx-order}.
Section \ref{subsec:32}  characterizes cutting length sets $\sK(d, n) = \{k : d = \floor{n / k}\}$ when $d$ is a floor quotient of $n$,
and studies the effect of the involution $\J_n$.
In Section~\ref{sec:floor-multiples} we show the set of all floor multiples of a given integer $d$
forms a numerical semigroup, and  characterize the resulting class of numerical semigroups. 

\item
Section~\ref{sec:initial-intervals} studies  the structure of initial intervals $\ADset[1,n]$.
Section~\ref{subsec:interval-involution} shows that the $n$-floor reciprocal map $\J_n: d \to \floor{n/d}$
is an involution, and proves Proposition~\ref{prop:interval-complement}.
The interval divides into two sets $\ADsetsmall(n)$ and $\ADsetlarge(n)$, having at most one element in common, 
which are exchanged by $\J_n$.
Section~\ref{subsec:gaps-cutting-mults} shows a duality relating gaps between successive floor quotients and cutting multiplicities; 
their values interchange under the involution $\J_n$. 
Section \ref{subsec:poset-interval-large-small} studies 
$\ADsetsmall(n)$ and $\ADsetlarge(n)$ as posets with the floor quotient relation,
and the action of $\J_n$ on these posets.
 Section \ref{subsec:consecutive-intervals} relates the poset structures of consecutive intervals $\ADset[1,n-1]$ and $\ADset[1, n]$. 
Section \ref{subsec:count-incidences} counts  the number of incidences of the floor quotient poset in an interval $\ADset[1,n]$,
and Section \ref{subsec:count-chains} counts the number of chains in intervals.

\item 
 Section \ref{sec:interval-size} studies the failure of scale-invariance for the floor quotient poset. 
Section \ref{subsec:51a} shows the size of a floor quotient interval 
$\ADset [d,n]$ is not scale invariant. 
The deviation from scale-invariance is measured as a function of the width  $w(d,n) \coloneqq {n}/{d}$ 
of an interval, a scale-invariant quantity.
 Section \ref{subsec:52a} presents universal upper bounds for the size of an interval in terms of its width,
 yielding Theorem \ref{thm:interval-bound0}. 
Section \ref{subsec:53a}  obtains lower bounds for its size, necessarily as functions of the parameters $n$ and $d$
separately.

\item
Section \ref{sec:mobius} treats the M\"{o}bius function of the floor quotient poset, starting with general facts.
Section \ref{subsec:62} shows that 
$\muAD(d,n)$ is expressible in terms of the classical M\"{o}bius function  when $d \geq \sqrt{n}$, as stated in Theorem \ref{thm:upper-interval-mobius2}. 
 Section \ref{subsec:mobius-bounds} finds upper bounds on the size of the M\"{o}bius function $\muAD(d,n)$. 
 Section \ref{subsec:64} shows that $\muAD(1,n)$ changes sign infinitely many times. 
 Section \ref{subsec:diff-mobius} studies the differenced M\"{o}bius function
$\Delta\mu_1(n) := \muAD(1,n)- \muAD(1, n-1)$ and proves the recursion in Theorem~\ref{thm:intro-mobius-recursion}.   
Section \ref{subsec:66} shows that the differenced M\"{o}bius vanishes at many values of $n$.
 Section \ref{subsec:mobius-data} presents numerical data. 
\end{itemize}

  Section~\ref{sec:concluding} raises four directions for further investigation.
\begin{itemize}
\item Asymptotic growth of the floor quotient M\"{o}bius
functions;
\item  Incidence algebra of the
floor quotient order;
\item   Generalized floor quotient partial orders:
the family of $a$-floor quotient orders $\ADa$ for integer $a \ge 1$,
with the floor quotient order $\ADset$ being the case $a=1$;
\item Approximate divisor orders interpolating
towards the additive total order.
\end{itemize} 

   \subsection*{Acknowledgments}
We thank the reviewer for helpful comments and improvements.
The first author was partially supported by  NSF grant DMS-1701576, 
by a Chern Professorship at MSRI in Fall 2018,
and by a 2018 Simons Fellowship in Mathematics, award 555520, in January 2019.
MSRI was supported in part by NSF award 1440140.
The second author was partially supported by
NSF grant DMS-1600223,
a Rackham Predoctoral Fellowship,
and an AMS--Simons Travel Grant.

%
%
\section{Floor quotient order versus divisor order}\label{sec:2} 

 In this section we  compare and contrast  the structure of  the floor quotient order with that of  the 
 additive order and the divisor order.  
We state relevant results, deferring proofs to subsequent sections.
To state results we set notation  for intervals of partial orders. 
Given a partial order $\sP = (\NNplus, \preccurlyeq_\sP)$ on the positive integers,
we let $\sP[m,n]$ denote the {\em  interval}  bounded by $m$ and $n$ in $\sP$, i.e.
\begin{equation}
\sP[m,n] = \{d \in \NNplus : m \preccurlyeq_\sP d \preccurlyeq_\sP n \}.
\end{equation}
We consider $\sP[m,n]$ as a poset with the induced order relation $\preccurlyeq_\sP$, unless otherwise specified.
The interval $\sP[m,n]$ is nonempty if and only if $m\preccurlyeq_\sP n$. 

For the additive order $\sA$, the divisor order $\sD$, and the floor quotient order
$\sQ$ on $\NNplus$ 
we have the set inclusions of initial  intervals
\[
\sD[1, n] \subseteq \ADset[1,n] \subseteq \sA[1, n] = \{1, \ldots, n\}.
\]
The additive order restricted to $\ADset[1,n]$ 
is a totally ordered linear extension of the floor quotient order.
The floor quotient order restricted to  $\sD[1,n]$
coincides with  the divisor order.

%
%
\subsection{The $n$-floor reciprocal map $\J_n$}
\label{subsec:20} 

The major  properties of floor quotients concern the action of its defining function
on initial intervals $\ADset[1,n]$.

\begin{defi}
For $n \ge 1$, the  {\em $n$-floor reciprocal  map} $\J_n: \sA[1,n] \to \sA[1,n]$ is given by
\begin{equation}\label{eqn:Phi-n}
\J_n(k) = \floor{ \frac{n}{k} }
  \quad \text{for} \quad 1\le k\le n.
\end{equation} 
By  definition of the floor quotient relation,
the floor quotient  interval  $\ADset[1,n]$ is precisely
the  range of  $\J_n$ on the domain 
 $\sA[1,n] = \{1,\ldots,n\}$. 
\end{defi} 

 An important feature of the $n$-floor reciprocal map $\J_n$ is that it leaves invariant each of the three domains 
$\sD[1,n]$,  $\ADset[1,n]$, and $\sA[1,n]$.
 The map $\J_n$ may therefore be iterated
on each of these domains.

We consider the action of this map $\J_n$  
on the intervals of the additive, multiplicative and floor quotient orders. 
\begin{enumerate}[(a)]
\item
The map $\J_n: \sD[1,n] \to \sD[1,n]$  on divisors of $n$ 
 is order-reversing for the divisor order.
We have $\J_n(d) = {n}/{d}$ for all $d \in \sD[1,n]$, 
and conversely $\J_n(d) = n/d$ only if $d \in \sD[1,n]$.
Furthermore, $\J_n$ is an involution on $\sD[1,n]$,
\begin{equation}
  \J_n^{\circ 2}(d) = d
\quad\text{for all}\quad d \in \sD[1,n],
\end{equation}
 hence $\J_n$ is 
 a bijection on $\sD[1,n]$.

 \item  
 The map $\J_n: \ADset[1,n] \to \ADset[1,n]$ 
is {\em never-order-preserving}
for the floor quotient order, meaning:
\begin{equation}
\text{if} \quad  d \AD e  \quad \text{then} \quad \J_n(d) \notAD \J_n(e),
 \end{equation} 
 i.e. either $\J_n(e) \AD \J_n(d)$ or else these two values are incomparable.
Furthermore $\J_n$ is an involution on $\ADset[1,n]$,
\begin{equation}
\label{eq:J-involution-fquo}
\J_n^{\circ 2}(d) =d \quad\text{for all}\quad d \in \ADset[1,n]
\end{equation}
 hence $\J_n$ is a bijection on  $\ADset[1,n]$.
 The floor quotient interval $\ADset[1,n]$ is  
 the maximal subset of $\{1, \ldots, n\}$ on which $\J_n$ acts as an involution,
 see Lemma \ref{lem:reciprocal-involution}.

\item
The map  $\J_n: \sA[1,n] \to \sA[1,n]$ is order-reversing for the additive order. 
 Its range is the floor quotient interval $\ADset[1,n]$.
Furthermore
\[
\J_n^{\circ 2}(k) \geq k \quad \mbox{for all} \quad k  \in \sA[1,n];
\]
 see Lemma \ref{lem:floor-reciprocal-cutting}.
 \end{enumerate}

We will  term the map $\J_n$ a {\em complementation map} specifically when its domain is $\ADset[1,n]$,
emphasizing the involution property.

%
%
\subsection{Lattice and rank properties}
\label{subsec:21}

The {divisor partial order} $\sD \coloneqq (\NNplus, \divides)$ is 
a  {\em distributive lattice} in the sense of Birkhoff. 
That is, it  has a  well-defined join function $\vee$ (least common multiple), 
and a well-defined meet function $\wedge$ (greatest common divisor),
which obey the distributive law
$x \wedge (y \vee z)= (x \wedge y) \vee (x \wedge z)$.
It also has a well-defined rank function,
 where the rank of an element $d$ is the length of any
set of covering relations to the minimal element $1$; 
the element $1$ is assigned rank $0$.
Each interval of the divisor partial order
is itself a ranked distributive lattice.

In contrast, the  floor quotient poset 
$\ADset  = (\NNplus, \AD)$
is not a lattice, and does not have a rank function.

\begin{enumerate}[(1)]
\item
 The floor quotient poset $\ADset  =(\NNplus, \AD) $ contains the 
 four element subposet on $\{2, 3, 6,7\}$ with Hasse diagram pictured
 in Figure \ref{fig:1}, in which all pictured edges are covering edges.
 

 \begin{figure}[h]  
 \[ \begin{tikzcd}
  6 \ar[dash]{d} \ar[dash]{dr} & 7 \ar[dash]{d} \ar[dash]{dl} \\
  2 & 3
 \end{tikzcd} \]
 \caption{Hasse diagram for the floor quotient relation on 
 $\{ 2, 3, 6,7\}$.} 
 \label{fig:1} 
 \end{figure}
 

 In Figure~\ref{fig:1}, elements $6$ and $7$ are both minimal common upper bounds for $2$ and $3$;
both $2$ and $3$ are maximal common lower bounds for $6$ and $7$.
 It follows that  the  floor divisor poset on $\NNplus$ does not  have unique least upper bounds 
 and does not have unique greatest lower bounds.  
 Thus it is neither a join semi-lattice nor a meet semi-lattice, hence not a lattice. 

\item
The floor quotient poset  $\ADset = (\NNplus, \AD)$ does not have a rank function. 
For any ranked poset  $P = (P, \preccurlyeq)$ with a minimal element, each finite interval  
of the  poset inherits  a rank function. 
But the  floor divisor poset $\ADset =(\NNplus, \AD)$ 
contains intervals not having a rank function. The  
 intervals $\ADset[1,9], \ADset[1,16]$
 and $\ADset[2, 20]$ each
  do not admit a rank function, as shown in Figure \ref{fig:2}.
  
\begin{figure}[h]
\centering
\raisebox{-0.5\height}{
 \begin{tikzpicture}[yscale=0.7]
 	\node (9) at (0,4) {$9$};
 	\node (4) at (-1,3) {$4$};
 	\node (3) at (1,2) {$3$};
 	\node (2) at (-1,1) {$2$};
 	\node (1) at (0,0) {$1$};
 	
 	\draw (9) -- (4) -- (2) -- (1);
 	\draw (9) -- (3) -- (1);
 \end{tikzpicture}
}
 \qquad\qquad
\raisebox{-0.5\height}{
 \begin{tikzpicture}[yscale=0.7]
  \node (16) at (0,6) {$16$};
  \node (8) at (-1,5) {$8$};
  \node (5) at (1,4.5) {$5$};
  \node (4) at (-2,3) {$4$};
  \node (3) at (1,1.5) {$3$};
  \node (2) at (-1,1) {$2$};
  \node (1) at (0,0) {$1$};
  \draw (16) -- (8) -- (4) -- (2) -- (1);
  \draw (16) -- (5) -- (2);
  \draw[out=-10, in=60] (16) to (3);
  \draw (3) -- (1);
  \end{tikzpicture}
}
 \qquad\qquad
\raisebox{-0.5\height}{
 \begin{tikzpicture}
 	\node (20) at (0,4) {$20$};
 	\node (10) at (-1,3) {$10$};
 	\node (6) at (1,2) {$6$};
 	\node (5) at (-1.2,1.4) {$5$};
 	\node (4) at (-2,1) {$4$};
 	\node (2) at (0,0) {$2$};
 	
 	\draw (20) -- (10) -- (5) -- (2);
 	\draw[out=180, in=90] (20) to (4);
 	\draw (4) -- (2);
 	\draw (20) -- (6) -- (2);
 \end{tikzpicture}
}
 \caption{Hasse diagrams for intervals $\ADset[1,9], \ADset[1,16]$
 and $\ADset[2, 20]$.}\label{fig:2} 
\end{figure} 
 
 
\end{enumerate} 

%
%
\subsection{Scaling-invariance properties}\label{subsec:23} 

The finite intervals $\sD[d,n]$ of the divisor poset $\sD = (\NNplus, \divides)$ have  a scaling-invariance property:
for each pair of endpoints $(d, n)$
and each scaling factor $a \ge 1$,
 there is an order-isomorphism of  divisor order intervals
\begin{equation}\label{eqn:scaling-inv}
\sD[d, n] \simeq \sD[ad, an],
\qquad\text{via}\qquad
m \mapsto am . 
\end{equation}
For  an interval $I= \sD[d,n]$ with endpoints $(d,n)$ we  call  $w(I) = w(\sD[d,n]) = {n}/{d}$ the {\em width} of the interval $I$.
The width is a scale-invariant quantity for the divisor poset: 
any  two intervals having the same width are order-isomorphic.
In particular, divisor intervals with the same width have the same cardinality. 

The  intervals $\ADset[d, n]$ of the  floor quotient order break scaling-invariance. 
That is, given $d \AD n$, the order interval
$\ADset [d,n]$ is generally not order-isomorphic to $\ADset [ad, an]$;
in particular the  cardinality of 
$\ADset [ad, an]$ generally varies with $a\ge 1$.
\begin{enumerate}[(i)]
\item
If $ d \AD n$ but $d \nmid n$, then $\ADset[d, n]$ is nonempty, while $\ADset[ad, an]$ is empty  for all large enough $a$. 
(See 
 Theorem \ref{thm:interval-lower-bounds}.)

\item
 The size of $\ADset[ad, an]$ can sometimes become much larger than
 $\ADset[d, n]$ for suitable $a$.
 \end{enumerate}
The size of the intervals $\ADset[ad, an]$, as $a$ varies, gives a quantitative measure of the extent to which scale-invariance is broken. 
We obtain  an upper bound on the size $|\ADset [d, n]|$ 
as  a function of a scale-invariant quantity, the  {\em width} $w(\ADset [d,n]) = \frac{n}{d}$.

\begin{thm}\label{thm:interval-bound0} 
For fixed $(d,n)$ and any $a \ge 1$, the  floor quotient interval $\ADset[ad,an]$ has size bounded by
\begin{equation}
\label{eqn:size-upper-bound0} 
	| \ADset[ad,an]| \leq \frac{3}{2} \left( \frac{n}{d} \right)^{2/3}   .
\end{equation} 
\end{thm}

This result is proved as Theorem \ref{thm:interval-bound}. 
Examples  show this upper bound is  the correct order of magnitude.
For the case $n/d = m^3$ and $a = 1$, we have
$|\ADset[1,m^3] |= 2 m^{3/2} +O(1)$; see Corollary \ref{cor:interval-size}.
However  taking $n/d = m^3$ and $a = m$ we have 
$|\ADset[ad, an]| =|\ADset[m, m^4]| =  \frac{3}{2}m^2 +O(m)$; see Proposition \ref{prop:410}.

%
%
\subsection{Self-duality structures: divisor poset intervals}
\label{subsec:24} 

For the divisor poset $\sD = (\NNplus, \divides)$, 
each of its initial intervals has two important symmetry properties:
a {\em self-duality as a set}, and {\em self-duality as a poset}.

\begin{enumerate}
\item[(1)]
The divisor interval $\sD[1,n]$ 
admits an involution $\J_n(d) = n/d$ which exchanges the lower and upper endpoints.  
We call this a {\em self-duality as a set}.
We split the order interval $\Dset[1,n]$  in two almost complementary parts
\[
	\Dset[1,n] = \Dset^{-}(n) \bigcup \Dset^{+}(n),
\]
in which
\[
	\Dset^{-}(n) = \{ d \in \Dset[1,n]: d \le \sqrt{n} \}
	\quad
	\text{and} 
	\quad
	\Dset^{+}(n) =  \{ d \in \Dset[1,n] :  \J_n(d) \leq \sqrt{n} \}.
\]
It is clear that 
$\J_n (\Dset^{+}(n) ) = \Dset^{-}(n)$
 and $\J_n (\Dset^{-}(n))  = \Dset^{+}(n)$, as sets.

Moreover, we have
\[
\Dset^{+}(n) \bigcap \Dset^{-}(n) 
= \begin{cases}
\{s\} &  \text{if }\, n = s^2, \\
\emptyset & \text{if }\, n \neq s^2  \quad \text{for any }\, s \ge 1.
\end{cases} 
\]
If  $n=s^2$ then the common element $s$ is the unique fixed point 
of the involution $\J_n$. 

\item[(2)] 
The involution $\J_n(d) = n/d$ is an {order-reversing} bijection on $\sD[1,n]$:
if $d \divides  e$ in the divisor order on $\sD[1,n]$ then $ (n/e) \divides  (n/d)$. 
We call this {\em self-duality as a poset}.

With respect to the splitting $\sD[1,n] = \sD^-(n) \cup \sD^+(n)$, 
$\J_n$ defines an anti-isomorphism from $\sD^+(n)$ to $\sD^-(n)$, as posets, and vice versa.
In other words, $\sD^+(n)$ is an ``order-reversed'' copy of $\sD^-(n)$, as posets.

Since $\J_n$ defines an anti-isomorphism from the whole interval $\sD[1,n]$ to itself, the interval $\sD[1,n]$ 
is an ``order-reversed'' copy of itself.
We can restate this in the language of dual partial orders.
Define the dual order $\mid^{\ast} $ on $\Dset[1,n]$ by
\[ 
	d \mid^{\ast} e \quad  \text{if and only if} \quad e \mid d. \]
Then the {$n$-reciprocal map} 
$\J_n(d) = {n}/{d}$ defines an isomorphism 
$
(\Dset[1,n], \divides) \xrightarrow{\sim} (\Dset[1,n], \;\mid^{\ast}\; )
$
and vice versa.
\end{enumerate}

%
%
\subsection{Self-duality structures: floor quotient intervals}
\label{subsec:24b} 

The initial floor quotient intervals  $\ADset[1,n]$ retain the first  of these self-duality structures (as {\em sets}), but not the second (as {\em posets}).

First, the $n$-floor quotient map $\J_n: \ADset[1,n] \to \ADset[1,n]$ is an involution, 
and the set  $\ADset[1,n]$ is the largest subset of $\{1, \ldots, n\}$ on which $\J_n$ restricts to be an involution
(Lemma \ref{lem:floor-reciprocal-cutting}). 
Cardinal \cite[Proposition 4]{Cardinal:10} noted there is a complementation operation 
on the floor quotient set $\ADset[1,n]$ 
which shows the self-duality property at the level of sets. 
We include a proof in Proposition \ref{prop:interval-complement}.
  
%
%
\begin{prop}
[Cardinal] 
\label{prop:interval-complement1}
Given  an integer $n \geq 1$,
let 
\begin{equation*}
\ADsetsmall(n) = \{d : d \AD n,\,  d \leq \sqrt{n} \}
\qquad\text{and}\qquad
\ADsetlarge(n) = \{d : d \AD n,\, \floorfrac{n}{d} \leq \sqrt{n} \}.
\end{equation*}
\begin{enumerate}[(a)]
\item 
We have
$ 
\ADset[1,n] = \ADsetsmall (n) \bigcup \ADsetlarge(n).
$
The  map $\J_n: d \mapsto \floor{n/d}$ defines a bijection sending $\ADsetlarge(n)$ to
$\ADsetsmall(n)$ 
and to a bijection sending $\ADsetsmall(n)$ to
$\ADsetlarge(n)$. 

\item 
Moreover, letting $s= \floor{\sqrt{n}}$,  one has
$$
\ADsetsmall(n) = \{ 1, 2, 3, \ldots, s\},
\qquad\text{and}\qquad
\ADsetlarge(n) 
= \{ n,  \floor{\frac{n}{2} }, \floor{\frac{n}{3} }, \ldots, 
\floor{ \frac{n}{s} } \}.
$$
The intersection of  $\ADsetsmall(n)$ and $\ADsetlarge(n)$ is
\begin{equation*}
\ADsetsmall(n) \bigcap \ADsetlarge(n) = \begin{cases}
\{s\} &\text{if } \, s^2 \le n < s(s+1) \\
\emptyset &\text{if } \,   s(s+1)\le n < (s+1)^2.
\end{cases}
\end{equation*}
\end{enumerate}
\end{prop} 

The  self-duality at the set level in Proposition \ref{prop:interval-complement1} is broken at the level of the partial order: 
the involution $\J_n(\cdot)$ is not order-reversing on $\ADset[1,n]$.
This failure happens in a ``one-sided'' sense: the map $\J_n: \ADsetlarge(n) \to \ADsetsmall(n)$ is order-reversing, but the map $\J_n: \ADsetsmall(n) \to \ADsetlarge(n)$ is not.
The next result shows that the partial order
can only be broken by becoming incomparable; 
it can never preserve order.  
We let $d \perp_1 n$ mean that $d$ is incomparable with $n$
in the floor quotient order.

%
%
\begin{thm}
[Broken self-duality as posets on initial floor quotient intervals]
\label{thm:interval-complement2}
Consider the sets 
\begin{equation*}
\ADsetsmall(n) = \{d : d \AD n,\, d \leq {\sqrt{n}} \}
\qquad\text{and}\qquad
\ADsetlarge(n) = \{d : d \AD n,\, \J_n(d) \leq \sqrt{n} \}.
\end{equation*}
\begin{enumerate}[(a)]
\item
 If $d \AD e$ in $\ADsetlarge(n)$, then $\J_n(e) \AD \J_n(d)$ in $\ADsetsmall(n)$,
 so $\J_n: \ADsetlarge(n) \to \ADsetsmall(n)$ is order-reversing for the floor quotient order.

\item
 If $d \AD e$ in $\ADsetsmall(n)$, then either $\J_n(e) \AD \J_n(d)$ in $\ADsetlarge(n)$
or $\J_n(e) \perp_1 \J_n(d)$ in $\ADsetlarge(n)$,
 so $\J_n: \ADsetsmall(n) \to \ADsetlarge(n)$ is never-order-preserving
 for the floor quotient order.
\end{enumerate}
\end{thm} 

This theorem follows from results given in Proposition ~\ref{prop:interval-complement} and Corollary \ref{cor:410}. 
To see an example of the extent to which $\ADsetsmall(n)$ fails to be dual to $\ADsetlarge(n)$, see Figure~\ref{fig:hasse-168} in Section \ref{subsec:poset-interval-large-small}, which shows the Hasse diagram of the floor quotient interval for $n = 168$.

%
%
\subsection{Incidences of floor quotient intervals}
\label{subsec:25} 

The amount of symmetry breaking in the poset-level complementation symmetry
can be quantified in terms of the incidence structure of the floor quotient order for the sets $\ADsetlarge(n)$ and $\ADsetsmall(n)$. 

We count the number of incidences in floor quotient intervals.
We let $Z(S_1, S_2)$ denote the number of incidences between members of
subsets $S_1$ and $S_2$ in a partial order, and we let $Z(S) = Z(S,S)$.
The following result is proved as Theorem \ref{thm:interval-incidence-bound}.

\begin{thm}[Incidence counts in initial intervals]
\label{thm:intro-interval-incidence}
Given $n \ge 2$, let $\ADset[1,n]$ be an initial interval of the floor quotient order.
Then the following hold. 

\begin{enumerate}[(a)]
\item The total 
number of incidences in $\ADset[1,n]$  satisfies 
\begin{equation}\label{eq:total-edges0}
\Z(\ADset[1,n]) = \frac{16}{3} n^{3/4} + O \left( n^{1/2} \right).
\end{equation} 

\item The number of incidences in $\ADsetsmall(n)$ and $\ADsetlarge(n)$ satisfy the bounds
\begin{align}
\Z(\ADsetsmall(n) ) &= \frac{4}{3} n^{3/4}  + O \big(n^{1/2}\big) \label{eq:small-edges0} \\
\Z(\ADsetlarge(n) ) &= \frac{1}{2} n^{1/2}\log n  + (2 \gamma -1)n^{1/2} + O \big( n^{1/4} \big). \label{eq:large-edge0}
\end{align} 

\item The number of incidences between $\ADsetsmall(n)$ and $\ADsetlarge(n)$ satisfy the bounds
 \begin{equation}
 \label{eq:small-large-edges0}
\Z( \ADsetsmall(n), \ADsetlarge(n) )= 4 n^{3/4} +O\left(  n^{1/2} \log n\right)
 \end{equation}
 and $\Z(\ADsetlarge(n), \ADsetsmall(n)) = 0  \text{ or } 1.$

\end{enumerate}
\end{thm}

 Theorem \ref{thm:intro-interval-incidence} shows that the incidence structure on $\ADsetlarge(n)$ comprises 
 an infinitesimally small fraction of the total incidences in $\ADset[1,n]$, as $n \to \infty$.
 The crossover set $\Z( \ADsetsmall(n), \ADsetlarge(n) )$ contains asymptotically $3/4$ of the
 incidence relations in the partial order restricted to $\ADset[1,n]$.

%
%

We can quantitatively compare  floor quotient initial intervals with
the initial intervals of the divisor order and the additive order, 
measuring either the size of initial intervals (as sets)  
or the number of incidences (as a partial order). 

Comparing the sizes of initial intervals for the three partial orders, we have:
\begin{enumerate}
 \item divisor order  intervals have size
$|\sD[1,n]| = O( n^{\epsilon})$ for any fixed $\epsilon> 0$, as $n \to \infty$;
 \item floor quotient intervals have size $|\ADset[1,n]| = 2 n^{1/2} + O(1)$, as $n \to \infty$;
 \item additive order intervals have size $|\sA[1,n]| = n$, for any $n \ge 1$.
\end{enumerate}

Comparing total incidence counts  in initial intervals for the three partial orders, we have:
\begin{enumerate}
 \item divisor order intervals satisfy 
$\Z(\sD[1,n])= O( n^{\epsilon})$ for any fixed  $\epsilon> 0$, for $n \to \infty$;
 \smallskip
 \item floor quotient intervals satisfy $\displaystyle \Z(\ADset[1,n]) = \frac{16}{3} n^{3/4} + O(n^{1/2})$, for $n \to  \infty$;
 \smallskip
 \item additive  order  intervals satisfy $\displaystyle \Z(\sA[1,n]) = \frac12 n^2 + O(n)$, for $n \to \infty$.
\end{enumerate}


%
%
\section{Structural results}\label{sec:structure}

This section studies relations among the three parameters $n, k, d$ in the defining equation of
the floor quotient relation, $d= \floor{n/k}.$  It first presents
six equivalent characterizations of the floor quotient relation $d \AD n$,
and prove it defines a partial order. 
It then studies the set of cutting lengths $k$ certifying that 
$d$ is a floor quotient of $n$. 
Finally it shows that the set of floor multiples of a given $d$ is a numerical semigroup,
and determines generators of this semigroup. 

%
%

\subsection{Floor quotient characterizations} 
\label{subsec:31} 
The floor quotient property has several 
quite different looking 
characterizations. 

%
%
\begin{thm}[Characterizations of floor quotient relation]
\label{thm:equiv-properties}  
The following properties  on pairs $(d, n)$ of positive integers are equivalent: 
\begin{enumerate}[(1)]
\item {\em  (Cutting property)}   $\dd = \floor{ \frac{n}{k}}$ for some integer $k$. 

\item {\em (Covering property)}  
For the half-open unit interval $I_\dd= [\dd, \dd+1)$ of $\RR$, there is some integer
$k$ such that 
$k I_\dd$ covers $I_n$,
i.e.
$ k  [\dd, \dd+1) \supseteq [n,n+1).$

\item
{\em (Intersection property) }
For  the  half-open unit interval $I_\dd = [\dd, \dd+1)$ of $\RR$, there is some integer
$k $ such that $   k  [\dd, \dd+1) \cap [n, n+1) \ne \emptyset.$

\item
 {\em (Strong remainder property)}
There is a positive integer $k$ such that
$$
n = \dd k + r \quad\text{where the remainder $r$ satisfies } \,\, 0 \leq r < \min(\dd,k).
$$

\item
 {\em (Tipping-point property)}
There holds the strict inequality 
$ \floor{\frac{n}{\dd}} > \floor{ \frac{n}{\dd+1} }$.

\item
{\em (Reciprocal-duality property)}
$\dd = \floor{ \frac{n}{\floor{n/\dd}} }$.
\end{enumerate} 

\end{thm}

\begin{rmk}
Each of the characterizations  in Theorem \ref{thm:equiv-properties} provides 
a useful viewpoint on properties  of the floor quotient order. 
For example,
\begin{enumerate}[(i)]
\item
The cutting property, Theorem~\ref{thm:equiv-properties} (1),
makes it clear that the floor quotient relation refines the divisibility relation $(\NNplus, \divides)$.

\item
The covering property, Theorem~\ref{thm:equiv-properties} (2), makes it clear that the floor quotient relation is transitive.

\item
The strong reminder property, Theorem~\ref{thm:equiv-properties} (4), makes it clear that floor quotients of $n$ 
occur naturally in symmetric pairs $(d,k)$.

\item
The tipping-point property and involution-duality property, Theorem~\ref{thm:equiv-properties} (5)-(6), 
provide easy direct tests for whether $d$ is a floor quotient of $n$.
\end{enumerate}

\end{rmk} 

\begin{proof}
Let  $[x,y)$ 
denote a half-open interval 
of $\RR$.

 $(1) \Leftrightarrow (2).$ 
Suppose  $d= \floor{ \frac{n}{k} }$ for some $k$, which is property (1).
Then $n= kd +r $ with $0 \le r \le k-1$,
which implies
 $[n, n+1) \subseteq [kd, kd+ k) =k[d, d+1)$.
This is the covering property (2) for $(d,n)$. 
These steps are all reversible, 
so the properties are equivalent.

$(2) \Leftrightarrow (3).$ 
The direction $(2) \Rightarrow (3)$ is immediate. For 
$(3) \Rightarrow (2)$ if $k[d , d+1)$ contains some $x \in [n,n+1)$ then it necessarily
contains the half-open interval $[n, n+1)$, 
because $[kd, kd+k)$ is a half-open interval with integer endpoints.

$(1) \Leftrightarrow (5)$.
For positive integers $\dd,n$, 
we define the {\em cutting length set} of $(d,n)$ as
$$
\cutset{\dd}{n} = \{ k \in \NNplus : \dd = \floor{n / k} \}.
$$
Property (1) says that this set is nonempty.
We have 
$$
\cutset{\dd}{n} = \{ k \in \NNplus : \dd \leq n/k < \dd + 1\} .
$$
The lower bound $\dd \leq \frac{n}{k} $
is equivalent to 
$k \leq \frac{n}{\dd}
\Leftrightarrow
k \leq \floor{ \frac{n}{\dd} }$.
The upper bound $\frac{n}{k} < \dd + 1$
is equivalent to
$ \frac{n}{\dd +1} < k
\Leftrightarrow \floor{ \frac{n}{\dd +1} } < k$.
Therefore
\begin{equation} 
\label{eqn:Kd-bound} 
\cutset{\dd}{n} = \{ k \in \NNplus : \floor{ \frac{n}{\dd +1}} < k \leq \floor{ \frac{n}{\dd} } \},
\end{equation} 
so $\cutset{\dd}{n}$ is nonempty if and only if $ \floorfrac{n}{\dd+1} < \floorfrac{n}{\dd}$,
which is the tipping point property (5).

$(1) \Leftrightarrow (6)$.
The direction $(6) \Rightarrow (1)$ is clear.
Now suppose (1) holds. 
The proof  of  $(1) \Leftrightarrow (5)$
established  \eqref{eqn:Kd-bound}, showing that 
$$
d = \floor{ \frac{n}{k} } \quad\text{is equivalent to }\quad 
\floor{ \frac{n}{d+1} } < k \leq \floor{ \frac{n}{d} } .
$$
Property (1) says that  
$\cutset{\dd}{n} = \{ k \in \NNplus : \dd = \floor{n / k} \}$
is nonempty, hence 
 it must contain $k=\floor{\frac{n}{d}}$,
which is  the involution-duality property (6).

$(4) \Leftrightarrow (6)$.
Suppose property (4) holds, then clearly $(4) \Rightarrow (1)$, and 
we have already shown $(1) \Rightarrow (6)$, so property (6) holds.

Now suppose property (4) does not hold for $k$;
we will show  property (6) does not hold. 
For a given $(d,n)$, applying division with remainder gives  
$n = dk + r$
for nonnegative integers $k,r$ with  $0 \leq r < d$.
In particular, $k = \floor{n/d}$ and  $r = n - d\floor{n/d}$.
Property (4) does not hold exactly when the remainder $r$
satisfies $r \geq k$.
In that case, we have  
$$
\frac{n}{ \floor{n/d}}
= \frac{dk + r}{k}
= d + \frac{r}{k} \geq d + 1,
$$
which implies
$\floorfrac{n}{\floor{n/d}} \geq d+1,$
so that property (6) does not hold. 
\end{proof}

We now 
show that the  equivalent properties 
in Theorem \ref{thm:equiv-properties} define 
an approximate divisor partial order on $\NNplus$,
as stated in Theorem \ref{thm:approx-order2}. 
 

\begin{thm}\label{thm:approx-order2} 
 The floor quotient relation $ \AD $
 defines a partial order  $\ADset \coloneqq ( \NNplus, \AD)$.
 The partial order $\ADset$ is an approximate divisor order.
\end{thm} 

\begin{proof}
To show that $\ADset = (\NNplus, \AD)$
is a partial order, 
it must be checked that $\AD$ is reflexive, antisymmetric and transitive.
 Taking $k=1$ certifies that  $n \AD n$  so the relation $\AD$ is reflexive. 
 If  $d \AD n$ and $d\neq n$
 then it is clear that $d < n$ (additively),
so the relation $\AD$ is antisymmetric. 
Finally we verify the transitivity property, 
which  asserts that $ d \AD m $ and $m \AD n $
 implies $d \AD n $.  
By Theorem~\ref{thm:equiv-properties} (1) $\Leftrightarrow$ (2), the  assertions  $ d \AD m $ and $m \AD n $ imply
 there exist positive integers $k_1, k_2$ such that $[m, m +1) \subseteq k_1[d, d+1) $
 and $[n, n+1) \subseteq k_2[m, m+1)$. 
 It follows that  
 $[n, n+1) 
 \subseteq k_2 k_1[d, d+ 1)$,
 which certifies that  $(d, n)$ 
 satisfies Theorem~\ref{thm:equiv-properties} (2), proving transitivity.
 
 Finally we verify
 this  partial order on $\NNplus$ is an approximate divisor ordering. 
 We have $d \AD n$ implies $d \le n$, since $\floor{ \frac{n}{k} } \le n$ for any $k \geq 1$. 
 On the other hand if $d\divides n$,
 then $d = \frac{n}{k}$ for some integer $k$,
 hence $d \AD n $ since $d = \floorfrac{n}{k}$ for the same $k$.
\end{proof} 

%
%

\subsection{Cutting length sets}\label{subsec:32}


\begin{defi}\label{def:cut-length} 
The {\em cutting length set} $\cutset{d}{n}$ of positive integers $d$ and $n$ is the set
\begin{equation}\label{eq:cutset}
\cutset{d}{n} \coloneqq \{ k\in \NNplus :  \; d= \floor{ \frac{n}{k} } \} .
\end{equation} 
\end{defi}

In terms of the $n$-floor reciprocal map $\J_n(k) = \floor{n/k}$  \eqref{eqn:Phi-n}, we have
$$\cutset{d}{n} 
= \J_n^{-1}(d)
=\{ k \in \NNplus : \J_n(k) = d \}.
$$


\begin{lem}\label{lem:cutting-size}
The cutting length sets $\cutset{d}{n}$ for $d \in \ADset[1,n]$ have
the following properties.
\begin{enumerate}[(a)]
\item
The members of a cutting length set are given by
\[
k \in \cutset{d}{n}
\qquad\text{if and only if}\qquad 
\floorfrac{n}{d+1} < k \leq \floorfrac{n}{d}.
\]

\item
The cardinality of $\cutset{d}{n}$ 
is $|\cutset{d}{n}| = \floorfrac{n}{d} - \floorfrac{n}{d+1}$.

\item
As $d \in \ADset[1,n]$ varies, 
the sets $\cutset{d}{n}$ 
partition the additive interval $\sA[1,n]$. 
\end{enumerate}
\end{lem}
\begin{proof}
The first two statements are verified in the proof of Theorem~\ref{thm:equiv-properties}, (1) $\Leftrightarrow$ (5).
Statement (c) follows from statement (a).
\end{proof}

%
%
\begin{lem}\label{lem:floor-reciprocal-cutting}  
\hfill
\begin{enumerate}[(a)]
\item 
Suppose $d \AD n$.
The  cutting length  set $\cutset{d}{n}$
contains exactly one element  $k^*$ which 
is a floor quotient of $n$. 
This element $k^*$  is the  largest  element
in $\cutset{d}{n}$ in the additive  order. 
One has  $k^* = \J_n(d) $.

\item 
The map  $\J_n^{\circ 2}= \J_n \circ\ \J_n$  on 
the additive interval $\sA[1,n] = \{1,\ldots,n\}$ sends each $k$ 
to the least floor quotient $k^*$  that is
no smaller than $k$ in the additive order. 
The image $\J_n^{\circ 2}(k) = k^*$ is the largest member of the cutting length set  $\cutset{\floorfrac{n}{k}}{n}$.
\end{enumerate}
\end{lem}  

\begin{proof}
(a) By the equivalence Theorem~\ref{thm:equiv-properties} (6) $\Leftrightarrow$ Theorem~\ref{thm:equiv-properties} (1),
there is exactly one floor quotient of $n$ in the cutting length set $\cutset{d}{n}$,
assuming $\cutset{d}{n}$ is nonempty.
If $\cutset{d}{n}$ is nonempty, its unique floor quotient element 
must be $k^* = \floor{n/d}$.
The statement that $k^*$ is the largest element of $\cutset{d}{n}$
follows from the tipping point property (5) of Theorem~\ref{thm:equiv-properties}.

(b) 
For any $k\leq n$, it is straightforward to check that
$\J_n^{\circ 2}(k) \geq k$.
The value $\J_n^{\circ 2}(k)$ is a floor quotient of $n$,
since the image of $\J_n$ consists of floor quotients of $n$.
Moreover, $\J_n^{\circ 2}(k)$ must be the least floor quotient of $n$ no smaller than $k$,
because $\J_n^{\circ 2}$ is order-preserving in the additive order,
and acts as the identity on the floor quotients $\ADset [1,n]$.

To verify the claim that $\J_n^{\circ 2}(k)$ is order-preserving in the additive order, 
it suffices to observe that $\J_n$ is order-reversing in the additive order,
in the sense that
$$
j \leq k \qquad\Rightarrow\qquad 
\J_n(j) \geq \J_n(k).
$$
Therefore 
$j \leq k $
implies $\J_n^{\circ 2}(j) \leq \J_n^{\circ 2}(k)$
as claimed.
\end{proof} 

%
%
\subsection{Floor multiples}\label{sec:floor-multiples}  
Given an integer $d$ we  say $n$ is a {\em floor multiple} of $d$ if $d$ is a floor quotient of $n$.

\begin{defi}\label{def:36}
 The set $\ADmult(d)$ of floor multiples of $d$ is 
\begin{equation}\label{eq:floor-mult}
\ADmult(d) \coloneqq \{ n\in \NNplus : d \AD n \}. 
\end{equation} 
It is  the semi-infinite interval
$
\ADset[d,\infty) 
$
in the floor quotient partial order.
\end{defi}

We show below  that the set of  floor multiples $\ADmult(d)$ 
forms a numerical semigroup.

 A   {\em numerical semigroup} is a subset of $\NNplus$
that is closed under addition
and   contains all but finitely many elements of $\NNplus$.
 The {\em Frobenius number}  of a numerical semigroup is the largest integer
not belonging to the semigroup. 
Given positive integers $m_1,\ldots, m_d$, we let $\angles{m_1,\ldots,m_d}$ denote the subset of $\NNplus$ consisting of all nonnegative integer combinations of $m_1, \ldots, m_d$; this forms a numerical semigroup if and only if $\gcd(m_1,\ldots,m_d) = 1$.
 For general information on  numerical semigroups and of the
 Frobenius number of a numerical semigroup, see  
 Ram\'{i}rez Alfons\'{i}n \cite{RamirezA:05} and Assi and Garcia-S\'{a}nchez \cite{AssiG:16}. 
 
We now establish the  numerical semigroup property of floor multiples, 
and also determine their  Frobenius numbers and give  minimal generating sets for them.

%
%
%
%
\begin{thm}[Floor multiple numerical semigroups]
\label{thm:floor-multiple-struct}
 The set $\ADmult(d)$ of floor multiples of  $d$ is 
 a numerical semigroup.
 It has the following properties.
 \begin{enumerate}[(a)]
 \item 
 The largest integer not in the numerical semigroup $\ADmult(d)$
 is $d^2-1$.

 \item
 There are exactly $\frac{1}{2} (d-1)(d+2) = \frac{1}{2} (d^2+d - 2) $ positive integers not in $\ADmult(d)$.

 \item
 The  minimal generating set of $\ADmult(d)$
 has $d$ generators, with $\ADmult(d)= \angles{ \gamma_1, \gamma_2, \ldots, \gamma_{d}}$
 where
 $\gamma_i= i(d+1) - 1 =  (i-1)(d+1)+d$ 
 for $1 \le i \le d.$
  \end{enumerate}
\end{thm} 

Before proving the result, we give an example. 
%
%
\begin{exa}
\label{exa:57} 
The floor multiples of  $n=4$ are 
$\ADmult(4) = \{4\} \cup \{8,9\} \cup \{12, 13, 14\} \cup [16, \infty) $.
The numbers that are not 
in $\ADmult(4)$ are 
$\{ 1,2,3 \}\cup\{ 5, 6,7\}\cup\{ 10, 11\}\cup\{ 15\}.$
The minimal generating set of $\ADmult(4)$
is $\angles{4, 9, 14, 19}$. 
The Frobenius number of $\ADmult(4)$
is $15 = 4^2 - 1$.
\end{exa}

We prove Theorem \ref{thm:floor-multiple-struct} via two preliminary lemmas.

\begin{lem}\label{lem:multiples-add}
The set $\ADmult(d)$ of floor multiples of  $d$
 is closed under addition.
\end{lem}

\begin{proof}
Suppose $n_1$ and $n_2$ are  floor multiples of $d$.
Then there  exist positive integers $k_1$ and $k_2$ such that
\[ 
d = \floor{ \frac{n_1}{k_1}} = \floor{ \frac{n_2}{k_2} }.
\]
Without loss of generality, assume that
$ \frac{n_1}{k_1} \leq \frac{n_2}{k_2}$.
Then 
\[ \frac{n_1}{k_1} \leq \frac{n_1 + n_2}{k_1 + k_2} \leq \frac{n_2}{k_2},\]
so it follows that 
$ d = \floor{\frac{n_1+n_2}{k_1+k_2}}$.
Thus $n_1+n_2$ is a floor multiple of $d$ as desired.
\end{proof}

The next lemma determines the Frobenius set $\NNplus \smallsetminus \ADmult(d)$  of this numerical semigroup.
We show that the Frobenius set of $\ADmult(d)$ is contained
between $0$ and $d^2$.

%
\begin{lem}
\label{lem:multiples-elem} 
Fix $d \geq 1$ and let $\ADmult(d)$ be the set of floor multiples of $d$.

\begin{enumerate}[(a)]
\item 
If  $n\geq d^2$, then $n\in \ADmult(d)$.

\item 
If $1 \leq n < d^2$, 
let $n = dk + j$ for integers $k,j$ satisfying $0 \leq k < d$
and $0 \leq j < d$.
Then $n \in \ADmult(d)$
if and only if $j < k$.
\end{enumerate}
\end{lem}
\begin{proof}
(a) Suppose  $n \ge d^2$. 
Then we  can express
$n = k d+ j$ with 
$0  \le j < d$
and 
$k = \floorfrac{n}{d} \ge d$. 
We 
then have
$
\floor{ \frac{n}{k} } = d + \floor{ \frac{j}{k} } = d
$
because $j < d \le k$.
This shows that $d \AD n$.

(b)
Suppose $n = kd + j$ where $0 \leq j,k < d$.
If $k=0$, then $n = j < d$ so $d \notAD n$.
In this case $n \not\in \ADmult(d)$ and $j \geq k$.

Now suppose that $k \geq 1$.
If $j < k$, 
then  we have $\floorfrac{n}{k} = d + \floorfrac{j}{k} = d$,
so $d \AD  n$.
On the other hand if $j \geq k$, 
then $\floorfrac{n}{k} = d + \floorfrac{j}{k} \geq d+1$,
while 
$\floorfrac{n}{k+1}  = \floor{d - \frac{d-j}{k+1}} < d$.
This shows there is no integer $k$ such that $ \floorfrac{n}{k} =d$,
so $d \notAD n$.
\end{proof}

\begin{rmk}\label{rmk:46} 
Lemma~\ref{lem:multiples-elem} can alternatively be verified using the strong remainder property, Theorem~\ref{thm:equiv-properties} (4).
\end{rmk}

\begin{proof}[Proof of Theorem \ref{thm:floor-multiple-struct}]
Lemma~\ref{lem:multiples-add} shows $\ADmult(d)$ is an additive semigroup and Lemma~\ref{lem:multiples-elem} (a)
then shows it  is a  numerical semigroup.

(a) 
The condition in Lemma~\ref{lem:multiples-elem} (b) implies that $d^2 -1 = d(d-1) + (d-1)$ 
is not in $\ADmult(d)$.
On the other hand, Lemma~\ref{lem:multiples-elem} (a) implies that $\ADmult(d)$ contains all integers greater than $d^2 - 1$.

(b) 
To count the size of $\NNplus \setminus \ADmult(d)$,
we apply the condition in Lemma~\ref{lem:multiples-elem} (b).
There are $d-1$ positive integers less than $d$ which are all not in $\ADmult(d)$.
For $n$ in the range $k d \le n < (k +1) d$, 
for $k = 1,2,\ldots, d-1$, 
there  are exactly $d- k$ integers 
not in $\ADmult(d)$,
namely $n = k d + j$
with 
$k \leq j \le d-1$. 
The total size of $\NNplus \setminus \ADmult(d)$ is
\[
	d-1 + \sum_{k=1}^{d-1} (d - k) = d-1 + \frac12(d-1)d = \frac{1}{2} {(d-1)(d+2)} .
\]

(c)
Each generator $\gamma_i = id + (i-1)$ is in $\ADmult(d)$ 
by Lemma~\ref{lem:multiples-elem},
so we have the containment 
$\angles{ \gamma_1,\ldots, \gamma_d } \subset \ADmult(d)$.
The reverse containment $\ADmult(d) \subset \langle \gamma_1,\ldots, \gamma_d \rangle$
 holds because 
the generating set
$\gamma_1,\gamma_2,\ldots,\gamma_d$ contains the minimal element of $\ADmult(d)$ in each residue class modulo $d$,
which also follows from Lemma~\ref{lem:multiples-elem}.

It remains to prove  that $\gamma_1,\ldots,\gamma_d$ is  minimal as a generating set  
of $\ADmult(d)$,
i.e. that no strict subset of $\{\gamma_1,\ldots, \gamma_d\}$ generates $\ADmult(d)$. 
Since $\gamma_1 < \gamma_2 < \cdots < \gamma_d$,
it suffices to show that $\gamma_i$ is not contained in the numerical semigroup $\langle \gamma_1, \ldots, \gamma_{i-1} \rangle$.

Suppose for the sake of contradiction that $\gamma_i\in \langle \gamma_1, \ldots, \gamma_{i-1} \rangle$.
Then for some $\ell < i$, 
we have 
$
\gamma_i- \gamma_\ell \in \langle \gamma_1, \ldots, \gamma_{i-1} \rangle.
$
However, 
we have the inclusion
$\langle \gamma_1, \ldots, \gamma_{i-1} \rangle \subset \ADmult(d)$,
and 
Lemma~\ref{lem:multiples-elem} implies that 
\[
\gamma_i - \gamma_\ell  
= (i-\ell)d - (i-\ell)
 \not\in \ADmult(d) .
\]
This contradiction implies that $\gamma_i \not \in \langle \gamma_1, \ldots, \gamma_{i-1} \rangle$,
so $\gamma_1,\ldots,\gamma_d$ is the minimal  generating set  
of $\ADmult(d)$ as claimed.
\end{proof}

%
%
\section{Floor quotient initial intervals \texorpdfstring{$\ADset[1,n]$}{Q[1,n]}}
\label{sec:initial-intervals}

The initial intervals $\ADset[1,n]$ of the floor quotient order are preserved as a set by the  $n$-floor reciprocal map $\J_n$.
This interval subdivides into a set union of  two pieces $\ADsetsmall(n)$ and $\ADsetlarge(n)$
having very different properties.

%
%
\subsection{Complementation action  $\J_n$ on $\ADset[1,n]$}\label{subsec:interval-involution} 

%
\begin{lem}\label{lem:reciprocal-involution} 

The $n$-floor reciprocal map $\J_n(d) = \floor{n/d}$  acts as an involution when  restricted to the domain
$\ADset[1,n]$ of all floor quotients of $n$, i.e.
\begin{equation}\label{eq:37a}
\J_n^{\circ 2} (d) = d \quad \mbox{for all} \quad  d \in \ADset[1,n].
\end{equation}
The set $\ADset[1,n]$ is the largest subset of $\{1,\ldots,n\}$ on which $\J_n$ acts as an involution. 
\end{lem}
\begin{proof}
Here \eqref{eq:37a} is  the involution-duality property (6) of Theorem~\ref{thm:equiv-properties};
in particular, we showed that Theorem~\ref{thm:equiv-properties} (1) $\Rightarrow$ Theorem~\ref{thm:equiv-properties} (6).
The reverse implication 
Theorem~\ref{thm:equiv-properties} (6) $\Rightarrow$ Theorem~\ref{thm:equiv-properties} (1)
shows that $\J_n$ is not an involution on any set strictly larger than $\ADset[1,n]$.
\end{proof}

The involution property  of $\J_n$ permits
separation  of $\ADset[1,n]$  into two halves, consisting of ``small floor quotients'' and ``large floor quotients,''
which may overlap in one element. Let
\begin{equation}
\label{eq:interval-halves}
\ADsetsmall(n) = \{d : d \AD n,\, d \leq \sqrt{n} \}
\qquad\text{and}\qquad
\ADsetlarge(n) = \{d : d \AD n,\, \J_n(d) \leq \sqrt{n} \}.
\end{equation}

The following  result 
 gives a set-theoretic complementation operation
relating the sets $\ADsetsmall(n)$ and $\ADsetlarge(n)$ under the map $\J_n$.
It is due to Cardinal \cite[Proposition 4]{Cardinal:10} (in a different notation). 

%
%
\begin{prop}
[Almost complementation for initial floor quotient intervals]
\label{prop:interval-complement}
Given  an integer $n \geq 1$,
let $\ADsetsmall(n)$ and $\ADsetlarge(n)$ denote the sets \eqref{eq:interval-halves}.
\begin{enumerate}[(a)]
\item
We have
$ 
\ADset[1,n] = \ADsetsmall (n) \bigcup \ADsetlarge(n).
$
The  map $\J_n: k \mapsto \floor{n/k}$  restricts to a bijection sending $\ADsetlarge(n)$ to
$\ADsetsmall(n)$ 
and to a bijection sending $\ADsetsmall(n)$ to
$\ADsetlarge(n)$. 

\item
Moreover, letting $s= \floor{\sqrt{n}}$,  one has
$$
\ADsetsmall(n) = \{ 1, 2, 3, \ldots, s\},
\qquad\text{and}\qquad
\ADsetlarge(n) 
= \{ n,  \floorfrac{n}{2}, \floorfrac{n}{3}, \ldots, 
\floorfrac{n}{s} \}.
$$
The intersection of  $\ADsetsmall(n)$ and $\ADsetlarge(n)$ is
\begin{equation*}
\ADsetsmall(n) \bigcap \ADsetlarge(n) = \begin{cases}
\{s\} &\text{if } \, s^2 \le n < s(s+1) \\
\emptyset &\text{if } \,   s(s+1)\le n < (s+1)^2.
\end{cases}
\end{equation*}
\end{enumerate}
\end{prop}

%
%
\begin{rmk} \label{rmk:involution-fixed-point} 
\begin{enumerate}[(i)]
\item The map $\J_n$ from $\ADset[1,n]$ to itself
 is order-reversing for the additive order, but 
 is not order-reversing for the floor quotient relation $\AD$.
Applied to elements $d \AD e$ in $\ADset[1,n]$,  either $\J_n(e) \AD \J_n(d)$ or $\J_n (e) \perp_1 \J_n (d)$.
 Corollary \ref{cor:410} gives a more precise statement.

\item Proposition  \ref{prop:interval-complement} (b) determines the fixed points of 
the floor reciprocal map $\J_n$ on $\ADset [1,n]$. 
It has a unique fixed point $s = \floor{\sqrt{n}}$
if $s^2 \le n < s(s+1)$, and it has no fixed point if
$s(s+1) \le n < (s+1)^2$.
\end{enumerate}
\end{rmk}

\begin{proof}
[Proof of Proposition~\ref{prop:interval-complement}]
(a) 
We first verify that $\ADset[1,n] = \ADsetsmall(n) \bigcup \ADsetlarge(n)$.
Suppose that $d \AD n$ and 
$d \not\in \ADsetsmall(n)$.
By definition of $\ADsetsmall(n)$, we have
$d > \sqrt{n}$ which implies 
\[
\frac{n}{d} <\sqrt{n} 
\qquad\text{hence}\qquad
\J_n(d) \coloneqq \floorfrac{n}{d} < \sqrt{n} .
\]
Thus $d \in \ADsetlarge(n)$.

Next we verify that $\J_n$ defines a bijection from $\ADsetlarge(n)$ to $\ADsetsmall(n)$.
It is clear from the definition of floor quotient that 
$\J_n(k) \AD n$ for any $k\leq n$.
By construction of $\ADsetlarge(n)$,
$\J_n$ sends $\ADsetlarge(n)$ 
to  $\ADsetsmall(n) = \{ k : k\AD n,\,  k \leq \sqrt{n}\}$.
Moreover, $\J_n^{\circ 2}$ restricts to the identity map on $\ADsetlarge(n)$
by the involution-duality property, Theorem~\ref{thm:equiv-properties} (6).
Since the composition
\[
\ADsetlarge(n) \xrightarrow{\J_n} \ADsetsmall(n) 
 \xrightarrow{\J_n} \ADsetlarge(n)
\]
is a bijection,
this implies 
\[
\ADsetlarge(n) \xrightarrow{\J_n} \ADsetsmall(n)
\quad\text{is injective,}\quad\text{and}
\quad
\ADsetsmall(n) \xrightarrow{\J_n} \ADsetlarge(n)
\quad\text{is surjective}.
\]
Similarly, the fact that $\J_n^{\circ 2}$ restricts to the identity map on $\ADsetsmall(n)$ implies that
$$
\ADsetsmall(n) \xrightarrow{\J_n} \ADsetlarge(n)
\quad\text{is injective,}
\quad\text{and}
\quad
\ADsetlarge(n) \xrightarrow{\J_n} \ADsetsmall(n)
\quad\text{is surjective.}
$$
We conclude that $\J_n$ induces bijections 
$\ADsetsmall(n) \to \ADsetlarge(n)$
and 
$\ADsetlarge(n) \to \ADsetlarge(n)$
as desired.

(b) We first show that 
if $d \leq \sqrt{n}$, then $d \AD n$.
Applying division with remainder to the pair $(d,n)$, we have
$$
n = dk + r
$$
for positive integers $k,r$ satisfying $0 \leq r < d$
and $k = \floorfrac{n}{d}$.
By the strong remainder property, Theorem~\ref{thm:equiv-properties} (4), 
it suffices that  
$r < \min(d,k)$
to conclude $d \AD n$.
But since  $d\leq \sqrt{n}$,
\[
k = \floorfrac{n}{d} \geq \floor{\sqrt{n}} \geq d.
\]
Thus the bound $r < d$
also implies $r < k$.
This verifies that $d \AD n$ as claimed,
so 
\[
\ADsetsmall(n) =\{ d : d \leq \sqrt{n} \}
=  \{1,2,\ldots, \floor{\sqrt{n}} \}.
\]
In (a), we showed that $\J_n : \ADsetsmall(n)\to \ADsetlarge(n)$ defines a bijection. It follows that
\[
\ADsetlarge(n) = \{\floorfrac{n}{k} : k = 1,2,\ldots,\floor{\sqrt{n}}\}.
\]

It remains to check when the  two sets  $\ADsetsmall(n)$ and $\ADsetlarge(n)$ overlap.
Let 
\[
s = \floor{\sqrt{n}} = \max\{d : d \in \ADsetsmall(n) \}
\qquad\text{and}\qquad 
s^* = \floorfrac{n}{s} = \min\{ d : d \in \ADsetlarge(n) \} .
\]
It is straightforward to verify that 
$s\leq s^*$. 
If $s < s^*$, then $\ADsetsmall(n)$ and $\ADsetlarge(n)$ are disjoint,
which occurs exactly when $n \geq s(s+1)$.
On the other hand if $s = s^*$, then their intersection is exactly $\{s\}$,
which occurs when $n < s(s+1)$.
\end{proof}


\begin{cor}[Size of  floor quotient initial intervals]
\label{cor:interval-size} 
Let $n$ be a positive integer.
The initial floor quotient interval $\ADset[1,n]$ 
has size bounded by
\begin{equation*}
\label{eqn:initial-bound} 
2\sqrt{n} - 2 < | \ADset[1,n] | < 2\sqrt{n} .
\end{equation*} 
\end{cor}

\begin{proof}
This bound follows from Proposition \ref{prop:interval-complement}, which implies 
\[
| \ADset[1,n] | = \begin{cases}
2s - 1  & \text{if } s^2 \leq n < s(s+1)\\
2s &\text{if } s(s+1) \leq n < (s+1)^2,
\end{cases}
\]
where $s = \floor{\sqrt{n}}$.
In particular, the bound $2\sqrt{n}-2 < 2s-1$ for $n < s(s+1)$ follows from the inequality
\[
2\sqrt{x(x+1)} < 2x + 1
\qquad\text{for all }x > 0. \qedhere
\]
\end{proof}

%
%
\subsection{Spacing of floor quotients in intervals: gaps and multiplicities}
\label{subsec:gaps-cutting-mults} 

For fixed $n$, we introduce  two more  invariants to floor quotients
$\ADset[1,n]$, called {\em gaps} and {\em (cutting) multiplicities}.

\begin{defi}
\label{def:gap-mult} 
Let $n$ be fixed. 
For each $d \AD n$, the   {\em gap} $G(d, n)$ of  $d$ is 
$$ 
G(d, n) \coloneqq d- d^{-},
$$
where $d^{-}$ is the largest floor quotient of $n$ strictly smaller than $d$
in the additive ordering, using the convention that $1^{-} \coloneqq 0$.
\end{defi} 
\begin{defi}\label{def:gap-mult2}
 For each $d \AD n$ the  {\em multiplicity}
 ({\em cutting multiplicity})
  $K(d, n)$ of $d$ is 
$$
K(d, n) \coloneqq | \cutset{d}{n}|,\
$$ 
where $\cutset{d}{n}\coloneqq \{ k\in \NNplus :   d= \floor{ \frac{n}{k}} \}$
 is the {\em cutting length  set}
of all $k$ that  certify  $d$ is a floor quotient of $n$. 
\end{defi} 
\begin{figure}[h]
\centering
\begin{tikzpicture}[scale=0.6]
\draw (0,0) -- (11,0);
\draw (0,0) -- (0,11);

\draw[smooth,samples=100,domain=0.9:11] plot (\x,{(10/\x)});

\foreach \i in {1,...,11}{
	\foreach \j in {1,...,11}{
		\ifnum \numexpr \i*\j > 10 
			{\draw (\i,\j) circle (0.1);}
		\else
			{\draw[fill] (\i,\j) circle (0.1);}
		\fi
	};
};
\draw (1,10) circle (0.2);
\draw (2,5) circle (0.2);
\draw (3,3) circle (0.2);
\draw (5,2) circle (0.2);
\draw (10,1) circle (0.2);
\end{tikzpicture}
\caption{Integer points under the hyperbola $xy = 10$.
Each floor quotient $d \AD 10$ corresponds to a corner point $(\floor{{10}/{d}}, d)$.}
\label{fig:integer-hyperbola}
\end{figure}

The next result  shows that gaps and multiplicities interchange under the involution $\J_n$
 restricted to  the domain  $\ADset [1, n]$.

%
\begin{lem}[Gaps and multiplicities interchange under $\J_n$]
\label{lem:gap-multiplicity} 
Given $n \ge 1$, set $s= \floor{ \sqrt{n}}$. 
\begin{enumerate}[(a)]
\item The map $\J_n$ interchanges gaps and cutting multiplicities.
That is, for  each $d \AD n$, 
one has  
$G( \J_n(d), n) = K(d, n) $ and $K(\J_n(d), n) = G(d,n)$. 

\item If $d$  is in $\ADsetsmall(n)$, i.e.  $1\le d \le s$,
then $d$ has gap 
$G(d, n) = 1$
and  cutting multiplicity 
$K(d,n) =  \floor{ \frac{n}{d} }- \floor{ \frac{n}{d+1} }$.

\item
If $d$ is in $\ADsetlarge(n)$,
i.e.   $d  = \floor{ \frac{n}{k} }$ 
 for $1 \le k \le s$,
 then $d$ has gap 
$G(d, n) =  \floorfrac{n}{k} - \floorfrac{n}{k+1}$
and  multiplicity $K(d, n)= 1$.
\end{enumerate}
\end{lem}

\begin{rmk}\label{rmk:316} 
The assertions of Lemma \ref{lem:gap-multiplicity} are illustrated in Figure~\ref{fig:integer-hyperbola}.
In this figure  $\ADset[1,10] = \{ 1, 2, 3, 5, 10\}$ and $s= \floor{\sqrt{10}} = 3$.
We make the following observations about Figure \ref{fig:integer-hyperbola}. 
\begin{enumerate}[(i)]
\item
Each floor quotient $d$ of $10$ corresponds to a corner point 
$(\floorfrac{10}{d}, d)$
of the set of integer points under the hyperbola $xy = 10$.

\item
The gap  $G(d,10)$ at $d$ is equal to the vertical distance from $(\floorfrac{10}{d},d)$ to the next corner in the South-East direction.

\item
The multiplicity $K(d,10)$ at $d$ is equal to the horizontal distance from $(\floorfrac{10}{d},d)$ to the next corner in the North-West  direction.

\item 
The figure is symmetric under interchanging the $x$- and $y$-axes.
\end{enumerate}
For these observations to hold for $d = 1$ and $ n$, we consider the hyperbola figure as having two additional points $(\infty, 0)$ and $(0, \infty)$ on either side of the boundary points $(n, 1)$ and $(1,n)$.
\end{rmk}

\begin{proof}[Proof of Lemma \ref{lem:gap-multiplicity}] 
(a)
Let $d \in \ADset [1,n].$
From Lemma~\ref{lem:cutting-size}, we have $K(d,n) = \floor{\frac{n}{d}} - \floor{\frac{n}{d+1}}$.
If we verify that 
$G(d,n) = K(\J_n(d), n)$ for all $d\AD n$,
it would follow that
$G(\J_n(d),n) = K(\J_n^{\circ 2}(d), n) = K(d,n)$
by the involution-duality property of floor quotients of $n$.
Thus it suffices to show that for each $d\AD n$
 the gap $G(d,n) = d - d^-$ satisfies
\begin{equation}\label{eq:38}
G(d,n) = \floor{\frac{n}{\J_n(d)}} - \floor{\frac{n}{\J_n(d)+1}} .
\end{equation}

The assertion
$d = \floor{\frac{n}{\J_n(d)}} $ 
is the involution-duality property, given as Theorem~\ref{thm:equiv-properties} (6).
Recall that $d^-$ denotes the largest floor quotient of $n$ strictly smaller than $d$.
To be a floor quotient of $n$, we need $d^- = \floor{\frac{n}{k}}$ for some $k$, 
and to have $d^- = \floor{\frac{n}{k}} < d$ we need $k >  \J_n(d)$. 
We claim that 
\begin{equation}\label{eq:d-minus} 
d^- =  \floorfrac{n}{\J_n(d)+1}
= \floor{ \frac{n}{\floor{n/d} + 1}} 
\qquad \text{for any } d \in \ADset [1,n].
\end{equation} 
To verify the formula \eqref{eq:d-minus} it  suffices to check that the strict inequality $\floor{\frac{n}{\J_n(d)+1}} < d $ holds.
But the inequality
$$
\floor{\frac{n}{\J_n(d)+1}} < \floor{\frac{n}{\J_n(d)}} = d,
$$
is  exactly the tipping-point property, Theorem~\ref{thm:equiv-properties} (5), applied to the floor quotient $\J_n(d)$.
This verifies the claim \eqref{eq:d-minus}, so \eqref{eq:38} follows.

(b) 
Suppose $d \in \ADsetsmall(n)$.
By Proposition~\ref{prop:interval-complement} $\ADsetsmall(n) = \{1, 2, \ldots, s \}$,
so $d^- = d-1$,
and the gap at $(d,n)$ is $G(d,n) = d - (d-1) = 1$.

(c) 
Suppose $d \in \ADsetlarge(n)$.
By Proposition~\ref{prop:interval-complement}, we have $d = \J_n(k)$ for some $k \in \ADsetsmall(n)$.
Then it follows from (a) and (b) that
\[
K(d,n) = K(\J_n(k),n)
=  G(k, n) = 1
\]
and
\[
G(d,n) = G(\J_n(k),n) 
= K(k,n) = \floorfrac{n}{k} - \floorfrac{n}{k+1},
\]
as asserted.
\end{proof}

%
%
\subsection{Floor quotient poset restricted to $\ADsetlarge(n)$ and $\ADsetsmall(n)$}
\label{subsec:poset-interval-large-small} 

We  consider the  floor quotient order relation $\AD$ 
restricted to the sets  $\ADsetlarge(n)$ and $\ADsetsmall(n)$. 
In general neither $\ADsetlarge(n)$ nor $\ADsetsmall(n)$ is an interval of the floor quotient partial order.
The  set  $\ADsetsmall(n)$
has a minimal element $1$ and generally several maximal elements, while the set $\ADsetlarge(n)$ has a maximal element $n$
but generally has several minimal elements. 

The set $\ADsetsmall(n)$ of ``small floor quotients'' coincides with the additive interval 
$\sA[1, \floor{\sqrt{n}}]$. 
The  floor quotient order relation on $\ADsetsmall(n)$
is  just the restriction of the floor quotient order
to the {\em additive} initial  interval  $\sA[1, \floor{\sqrt{n}}]$.

The main result of this subsection shows that when restricted to  the ``large floor quotients'' $\ADsetlarge(n)$, 
the floor quotient order is order-isomorphic to  
the dual of the divisor partial order restricted to the additive  interval $\sA[1, \floor{\sqrt{n}} ]$.
Recall that the {\em dual} $P^{\ast}$ of a partially ordered set $P$ is the same underlying set with the order relation reversed.
We let $(\ADsetlarge(n))^{\ast}$ denote the dual partially ordered set for $\ADsetlarge(n)$; 
it has a minimal element $n$ and in general many  maximal elements.

\begin{thm}
\label{thm:upper-interval-struct}  
The floor quotient partial order  restricted to  the domain
$
\ADsetlarge(n) = \{ \floor{ n/k } \,: \, 1 \le k \le \floor{ \sqrt{n}} \}
$
is anti-isomorphic to the divisor partial order 
restricted to the additive interval $\sA[1, \floor{ \sqrt{n} }] = \{1,2,\ldots, \floor{\sqrt{n}}\}$,
via the map $\floorfrac{n}{k} \mapsto k$.
That is: for $1 \leq i,j \leq \floor{\sqrt{n}}$ we have
\begin{equation}\label{eq:dual-divisor}
\Big\lfloor{\frac{n}{j}}\Big\rfloor \AD \floor{\frac{n}{i}}
\qquad\text{if and only if}\qquad
i \divides j.
\end{equation}
\end{thm}

\begin{proof}
 Recall that the cutting length set is defined as
$\cutset{d}{n} = \{ k \in \NNplus : \floor{n/k} = d\}$.
Observe that
\begin{align*}
\cutset{\floor{\frac{n}{j}}} {\floor{\frac{n}{i}}} 
= \{k \in \NNplus :  \floor{ \frac1k \floorfrac{n}{i} } = \floorfrac{n}{j} \}.
\end{align*}
The identity \eqref{eq:floor-dilation-commute} states that
$
\floor{ \frac{1}{k} \floorfrac{n}{i} } = \floorfrac{n}{i k} ,
$
for all positive integers $i, k, n$,  
cf,  \cite[Lemma 6]{Cardinal:10} or \cite[Theorem 1.1]{LMR:16}.
It yields  
\begin{equation*}
k \in \cutset{\floor{\frac{n}{j}}} {\floor{\frac{n}{i}}}
\qquad\Leftrightarrow\qquad
\floorfrac{n}{j}
= \floor{ \frac1k \floorfrac{n}{i} } = \floorfrac{n}{ik} 
\qquad\Leftrightarrow\qquad
ik \in \cutset{\floor{\frac{n}{j}}} {n}.
\end{equation*}

Now suppose that
$
\floor{n/j} \AD \floor{n/i}
$
where $i,j \leq \sqrt{n}$.
Since $j \leq \sqrt{n}$, the cutting length set $\cutset{\floor{n/j}}{n}$
has cardinality $1$ by Lemma~\ref{lem:gap-multiplicity} (3),
so it contains only $j$.
By hypothesis, $\cutset{\floor{n/j}}{ \floor{n/i}}$
is nonempty so choose any $k \in \cutset{\floor{n/j}}{ \floor{n/i}}$.
Hence we have
$$
ik \in \cutset{\floorfrac{n}{j}}{n}  = \{j\}.
$$
It follows that $j = ik$, so $i \divides j$.

Conversely, suppose $i \divides j$.
Then $j = ik$ for some integer $k$,
and it follows that
$$
\floorfrac{n}{j} = \floorfrac{n}{ik}
= \floor{ \frac{1}{k} \floorfrac{n}{i}}
\qquad\Rightarrow\qquad
\floorfrac{n}{j}  \AD \floorfrac{n}{i} .
$$
(For this direction of implication, we do not need $i,j \leq \sqrt{n}$.)
\end{proof}

We can now  determine the action of the $n$-floor reciprocal map on the floor quotient order.

\begin{cor}[Action of $\J_n$ on floor quotient order]
\label{cor:410} 
Let $n$ be a positive integer. Then
\begin{enumerate}[(a)]
\item
 If $d \AD e$ in $\ADsetlarge(n)$, then $\J_n(e) \AD \J_n(d)$ in $\ADsetsmall(n)$,
 so $\J_n: \ADsetlarge(n) \to \ADsetsmall(n)$ is order-reversing. 

\item
If $d \AD e$ in $\ADsetsmall(n)$ and $d \neq e$, then 
$\J_n(d) \notAD \J_n(e)$ in $\ADsetlarge(n)$,
so $\J_n: \ADsetsmall(n) \to \ADsetlarge(n)$
is never-order-preserving.
\end{enumerate}
\end{cor}

\begin{proof}
The map $\J_n: \ADset[1,n] \to \ADset[1,n]$ is strictly order-reversing for the additive total order,
i.e. $d \leq e$ implies $\J_n(d) \geq \J_n(e)$.
Since the additive order $\leq$ is a refinement of $\AD$,  if $d \AD e$ with $d \ne e$  then $\J_n(d)  \notAD \J_n(e)$ so the
action of $\J_n$ is never-order-preserving in both cases (a) and (b). 
Thus (b) is proved.

In case (a) given $d \AD e \in \ADsetlarge(n)$,
it remains to show that  $\J_n(e) \AD \J_n(d)$ in $\ADsetsmall(n)$.
Proposition~\ref{prop:interval-complement} implies that $d = \floor{n/j} = \J_n(j)$ and $e = \floor{n/k} = \J_n(k)$ for integers $j,k$ satisfying
$1 \le k \le j \le \floor{\sqrt{n}}$.
By Theorem \ref{thm:upper-interval-struct},
$\floor{n/j} \AD \floor{n/k}$  implies that $k \divides j$.
Since $\AD$ refines the divisor order
by Theorem \ref{thm:approx-order2}, we have $k \AD j$.
Finally, we observe that 
$\J_n(d) = \J_n^{\circ 2}(j) = j$ and
$\J_n(e) = \J_n^{\circ 2}(k) = k$ by Lemma~\ref{lem:reciprocal-involution} so
$\J_n(e) \AD \J_n (d)$, which proves (a).
\end{proof}

%
\begin{exa} 
\label{ex:411bb} 
In Figure \ref{fig:hasse-168} we  illustrate  the order structure of $\ADset[1, 168] = \ADsetsmall(168) \bigcup \ADsetlarge(168)$.

\begin{figure}[h]
\centering
\begin{tikzpicture}[scale=0.8]
	\node (168) at (0,12) {$168$};
	\node (84) at (-2,11) {$84$};
	\node (56) at (1,11) {$56$};
	\node (42) at (-4,10) {$42$};
	\node (33) at (3,10) {$33$};
	\node (28) at (0,9) {$28$};
	\node (24) at (5,9) {$24$};
	\node (21) at (-6,8) {$21$};
	\node (18) at (-3.5,8) {$18$};
	\node (16) at (2,8) {$16$};
	\node (15) at (7,8) {$15$};
	\node (14) at (-2,7) {$14$};
	\node (12) at (-2,5) {$12$};
	\node (11) at (7,4) {$11$};
	\node (10) at (2,4) {$10$};
	\node (9) at (-3.5,4) {$9$};
	\node (8) at (-6,4) {$8$};
	\node (7) at (5,3) {$7$};
	\node (6) at (0,3) {$6$};
	\node (5) at (3,2) {$5$};
	\node (4) at (-4,2) {$4$};
	\node (3) at (1,1) {$3$};
	\node (2) at (-2,1) {$2$};
	\node (1) at (0,0) {$1$};
	
	\foreach \a/\b in {
		56/11, 
		42/8, 
		28/9,
		21/10, 
		18/9, 18/6, 
		15/5} {
		\draw[color=lightgray] (\a) to (\b);
	}
	\draw[color=lightgray, out=-110, in=110] (84) to (12);
	\draw[color=lightgray, out=-140, in=60] (56) to (8);
	\draw[color=lightgray, out=-30, in=100] (42) to (6);
	\draw[color=lightgray, out=-50, in=110] (33) to (11);
	\draw[color=lightgray, out=-90, in=50] (33) to (6);
	\draw[color=lightgray, out=-90, in=140] (28) to (5);
	\draw[color=lightgray, out=-130, in=20] (24) to (12);
	\draw[color=lightgray, out=-180, in=50] (24) to (8);
	\draw[color=lightgray, out=-50, in=180] (21) to (7);
	\draw[color=lightgray, out=-80, in=120] (21) to (4);
	\draw[color=lightgray, out=-150, in=30] (16) to (8);
	\draw[color=lightgray, out=-70, in=90] (16) to (5);
	\draw[color=lightgray, out=-110, in=90] (16) to (3);
	\draw[color=lightgray, out=-90, in=50] (15) to (7);
	\draw[color=lightgray, out=-20, in=140] (14) to (7);
	\draw[color=lightgray, out=-140, in=100] (14) to (4);

	\draw (168) -- (84) -- (42) -- (21);
	\draw (168) -- (56) -- (28) -- (14);
	\draw[out=-20, in=120] (168) to (33);
	\draw[out=-15, in=120] (168) to (24);
	\draw[out=-10, in=120] (168) to (15);
	\draw (84) -- (28);
	\draw[out=-30, in=130] (84) to (16);
	\draw (56) to (18);
	\draw (42) -- (14);
	\draw (33) -- (16);
	
	\draw (1) -- (2) -- (4) -- (8);
	\draw (1) -- (3) -- (6) -- (12);
	\draw (2) -- (5) -- (10);
	\draw (2) -- (6) -- (12);
	\draw[out=30, in=-170] (2) to (7);
	\draw[out=20, in=-130] (3) to (7);
	\draw (3) to (9);
	\draw (3) to (10);
	\draw[out=10, in=-120] (3) to (11);
	\draw (4) to (9);
	\draw (4) to (12);
	\draw[out=10, in=-140] (5) to (11);
	
\end{tikzpicture}
\caption{Hasse diagram for the interval $\ADset[1,168]$.}
\label{fig:hasse-168}
\end{figure}
\begin{enumerate}[(i)] 
\item
The elements $\ADsetsmall (168)= \{1, 2, \ldots, 11, 12\}$ comprise the small floor quotients $\ADsetsmall(168)$.
The order relations among these elements shows the floor quotient relation $\AD$ on an additive interval $\{1, \ldots, 12\}$.

\item
The elements $\ADsetlarge(168) = \{14, 15, 16, 18, 21, \ldots, 84, 168\}$ comprise the large floor quotients $\ADsetlarge(168)$.
The order relations among these elements is dual (anti-isomorphic) to the divisibility relation $d \divides e$ on an additive interval $\{1, \ldots, 12\}$.

\item
The covering relations between $\ADsetsmall(168)$ and $\ADsetlarge(168)$ are denoted by lighter gray lines.
Theorem \ref{thm:intro-interval-incidence}(c) asserts that as $n \to \infty$ the incidences in this set comprise $3/4$ of the incidence
relations of the interval $\ADset[1,n]$ in the partial order.
\end{enumerate}
\end{exa} 

%
%
\subsection{Consecutive floor quotient intervals $\ADset[1,n-1]$ and $\ADset[1, n]$}
\label{subsec:consecutive-intervals}

We compare the  structure of consecutive floor quotient intervals $\ADset[1,n-1]$ and $\ADset[1,n]$,
first   as partial orders, then as sets of integers.

The following result treats the change
in partial order structures.

\begin{thm}
\label{thm:consecutive-interval-posets}
Let $n \ge 2$ and set $s= \floor{\sqrt{n}}$. 
\begin{enumerate}[(a)]
\item 
Let 
$\iota^-_n: \ADsetsmall(n-1) \to \ADsetsmall(n)$
be the identity map $\iota^-_n(d) = d$.
If $n \neq s^2$,  then $\iota_n^-$ is an isomorphism of posets,
\[
(\ADsetsmall(n-1), \AD) \xrightarrow{\sim} (\ADsetsmall(n), \AD).
\]
If $n = s^2$, then $\iota_n^-$ is injective and is an isomorphism onto its image, 
and $s \in \ADsetsmall(n)$ is the unique element not in the image of $\iota^-_n$.

\item
Let 
$\iota_n^+: \ADsetlarge(n-1) \to \ADsetlarge(n)$
be the map that sends $\iota^+_n(\floorfrac{n-1}{k}) = \floorfrac{n}{k}$.
If $n \neq s^2$, then 
$\iota_n^+$ is an isomorphism of posets,
\[
(\ADsetlarge(n-1), \AD) \xrightarrow{\sim} (\ADsetlarge(n), \AD).
\]
If $n = s^2$, 
then $\iota_n^+$ is injective and is an isomorphism onto its image,
and $s \in \ADsetlarge(n)$ is the unique element not in the image of $\iota^+_n$.
\end{enumerate}
\end{thm}

\begin{proof}
\begin{enumerate}[(a)]
\item[(a)] 
This result follows from the characterization of $\ADsetsmall(n)$ in Proposition~\ref{prop:interval-complement} (b).
\item [(b)]
This result follows directly  from Theorem~\ref{thm:upper-interval-struct}, characterizing the poset structure of $\ADsetlarge(n)$.
\qedhere
\end{enumerate}
\end{proof}

We now  compare the consecutive intervals $\ADset[1,n-1]$ and $\ADset[1,n]$,
viewed  as sets.
The sets  $\ADset[1,{n-1}]$ and $\ADset[1,n]$ change between  $\ADsetlarge(n-1)$ and $\ADsetlarge(n)$,
adding new elements at each of the  divisors $e$ of $n$ with $e > \sqrt{n}$, and 
for such $e$, removing elements $e-1 \in \ADsetlarge(n-1)$.
These changes do not change the size of $\ADset[1,n-1]$ going to $\ADset[1,n]$. The increases by $1$ in size occur for
$n=s^2$ and $n=s(s+1)$, where   a  new element $s$ (resp. $s+1$)  is added, with no corresponding element removed.

These changes, summarized in the following lemma, 
underly  a recursion for computing the differenced floor quotient M\"{o}bius function
 appearing in Theorem \ref{thm:mobius-recursion}.

\begin{lem}
\label{lem:consecutive-interval-sets}
For any $n\geq 2$, the floor quotient interval $\ADset[1,n]$ satisfies
\[ 
\ADset[1,n] = \begin{cases}
	\ADset[1,n-1] \setminus \{d-1 : d \in \sD^+(n)\} \bigcup  \sD^+(n) \\
	\qquad\qquad\qquad\qquad\qquad\qquad  \text{if $n\neq s^2$ or $s(s+1)$ for any }s \geq 1, \\[1em]
	\ADset[1,n-1] \setminus \{d-1 : d \in \sD^+(n)\} \bigcup \sD^+(n) \bigcup \{s\} \\
	\qquad\qquad\qquad\qquad\qquad\qquad  \text{if } n = s^2 \text{ or }s(s+1), \\
\end{cases}
\]
where
$
\sD^+(n) = \{ d : d > \sqrt{n},\, d\divides n\}.
$
\end{lem}
\begin{proof}
We use  the decomposition $\ADset[1,n] = \ADsetsmall(n) \bigcup \ADsetlarge(n)$  in Theorem~\ref{prop:interval-complement}.
For any $n\geq 2$, the small floor quotients $\ADsetsmall(n)$ satisfy
\begin{equation}
\ADsetsmall(n) = \begin{cases}
	\ADsetsmall(n-1) &\text{if }n \neq s^2 \text{ for any }s, \\
	\ADsetsmall(n-1) \bigcup \{s\} &\text{if } n = s^2.
\end{cases}
\end{equation}
This follows directly from Proposition~\ref{prop:interval-complement} (b).

 Proposition~\ref{prop:interval-complement} (b) states that the large floor quotients $\ADsetlarge(n)$ satisfy
\[ 
  \ADsetlarge(n) = \{ \floorfrac{n}{k} : 1 \leq k \leq \floor{\sqrt{n}} \}.
\]
We use the fact 
\[
\floorfrac{n}{k} - \floorfrac{n-1}{k} = \begin{cases}
1 &\text{if } k \divides n,\\
0 &\text{otherwise}.
\end{cases}
\]
Therefore
\begin{equation}
\ADsetlarge(n) = \begin{cases}
	\ADsetlarge(n-1) \setminus \{d-1 : d \in \sD^+(n) \} \bigcup \sD^+(n) \\
	\qquad\qquad &\text{if } n\neq s^2 \text{ for any }s \geq 1,\\[1.0em]
	\ADsetlarge(n-1) \setminus \{d-1 : d \in \sD^+(n) \} \bigcup \sD^+(n) \bigcup \{s\} \\
	\qquad\qquad &\text{if } n = s^2.
\end{cases}
\end{equation}
Finally, 
Proposition~\ref{prop:interval-complement} (b)  states that
\[
 \ADsetsmall(n) \bigcap \ADsetlarge(n) = \begin{cases}
 \{ s\} &\text{if } s^2 \leq n < s(s+1) \\
 \emptyset &\text{if } s(s+1) \leq n < (s+1)^2,
 \end{cases}
\]
where $s = \floor{\sqrt{n}}$.
In particular, when $n = s(s+1)$
we have $s+1 \in \sD^+(n)$ and $s \in \{d -1 : d \in \sD^+(n)\}$; in this case we must add back $\{s\}$ to recover $\ADsetsmall(n)$.
\end{proof}

%
%
\begin{exa}
We treat a case where $n=s(s+1)$.
Consider $n = 12 = s(s+1)$ with $s = 3$.
Then 
\[
\ADsetsmall(11) = \{1,2,3\}
\qquad\text{and}\qquad
\ADsetlarge(11) = \{3,5,11\}
\]
while
\[
\ADsetsmall(12) = \{1,2,3\}
\qquad\text{and}\qquad
\ADsetlarge(12) = \{4,6,12\}.
\]
\end{exa}

%
%
\subsection{Counting incidences in floor quotient intervals}
\label{subsec:count-incidences}

The incidence algebra of a locally finite partially ordered set was formalized by Rota in \cite{Rota:64}. 
For a poset $(\sS,\preccurlyeq)$, it is the set of   functions $f: \sS \times \sS \to \QQ$
such that $f(x,y) = 0$ if $x \not\preccurlyeq y$.
It forms a vector space over $\QQ$. 
The algebra operations are pointwise sum and convolution, and
the M\"{o}bius function 
is the convolution inverse of the zeta function.

The {\em  zeta function} $\zeta(\cdot, \cdot)$ of a 
locally finite partially ordered set  $\sP= (\sS, \preccurlyeq) $ on domain $\sS$  is given by
\begin{equation}
\zeta(s, t) = \begin{cases} 
1 & \text{if} \, s \preccurlyeq t\\ 
0 & \text{otherwise}. 
\end{cases} 
\end{equation} 
The zeta function encodes information
 on the  number of incidence relations of $\preccurlyeq$ on subsets of
the domain $\sS$.
For subsets $\sS_1, \sS_2$ of $\sS$   we set
\begin{equation}\label{eq:cumulative12} 
\Z( \sS_1, \sS_2) \coloneqq \sum_{ s \in \sS_1}\sum_{ t \in \sS_2} \zeta(s, t).
\end{equation} 
If $\sS_1= \sS_2$ then we  write 
\begin{equation}\label{eq:cumulative11} 
\Z(\sS_1) \coloneqq \Z( \sS_1, \sS_1) = \sum_{ s, t \in \sS_1} \zeta(s, t).
\end{equation} 

The following result counts (non-strict)  incidences of the floor quotient order on
its initial intervals $\ADset[1,n]$,
as well as on $\ADsetsmall(n)$ and $\ADsetlarge(n)$.
The value $\Z(\ADset[1,n])$ is the vector space dimension of its incidence algebra.

\begin{thm}[Incidence counts in initial intervals]
\label{thm:interval-incidence-bound}
Given $n \ge 2$, let $\ADset[1,n]$ be an initial interval of the floor quotient order.
Then the following hold. 

\begin{enumerate}[(a)]
\item The total 
number of incidences in $\ADset[1,n]$  satisfies 
\begin{equation}\label{eq:total-edges}
\Z(\ADset[1,n]) = \frac{16}{3} n^{3/4} + O \left( {n}^{1/2} \right).
\end{equation} 

\item The sets $\ADsetsmall(n)$ and $\ADsetlarge(n)$ satisfy the bounds
\begin{align}
\Z(\ADsetsmall(n) ) &= \frac{4}{3} n^{3/4}  + O \big(n^{1/2}\big) \label{eq:small-edges} \\
\Z(\ADsetlarge(n) ) &= \frac{1}{2} n^{1/2}\log n  + (2 \gamma -1)n^{1/2} + O \big( n^{1/4} \big). \label{eq:large-edge}
\end{align} 

\item The incidences between $\ADsetsmall(n)$ and $\ADsetlarge(n)$ satisfy the bounds
\begin{equation}
 \label{eq:small-large-edges}
\Z( \ADsetsmall(n), \ADsetlarge(n) )= 4 n^{3/4} +O\left(  n^{1/2} \log n\right)
\end{equation}
and $\Z(\ADsetlarge(n), \ADsetsmall(n)) = 0  \text{ or } 1$.
 
\end{enumerate}
\end{thm} 
\begin{rmk} \label{remk:incidences} 
The bounds in Theorem \ref{thm:interval-incidence-bound}
count  all incidences including the trivial ones $x \AD x$.
 One can alternatively count strict incidences 
using the {\em incidence function} $n(x,y) = \zeta(x,y) - \delta(x,y)$
introduced by Rota \cite[Sec. 3]{Rota:64}, where $\delta(x,y)=1$ if $x=y$ and is $0$ otherwise.
Since $\ADset[1, n]$ has cardinality
$O( \sqrt{n})$,  estimates similar to (a)--(c) hold for counting strict incidences,
with the analogue of \eqref{eq:large-edge} requiring  a correction to the coefficient of the $n^{1/2}$-term.
\end{rmk} 

\begin{proof}[Proof of Theorem \ref{thm:interval-incidence-bound}]
By Corollary  \ref{cor:interval-size}, for each element $d \in \ADset[1,n]$ there are $2 \sqrt{d} + O(1)$ values $d'\in \ADset[1,n]$
having  $d' \AD d$.

Treating $\ADsetsmall(n)$,  since $d \in \ADsetsmall (n)$
are exactly $d \in \{1, \ldots, \floor{\sqrt{n}}\}$, we obtain
\[
	\Z(\ADsetsmall (n)) = \sum_{d=1}^{\floor{\sqrt{n}} } \left( 2 \sqrt{d} + O(1) \right).
\]
Replacing the summand by an integral 
on each unit subinterval $[d-1,d]$ 
introduces an error of $O \left( \frac{1}{\sqrt{d}} \right)$, 
yielding
\begin{align}\label{eq:small-bound} 
	\Z(\ADsetsmall (n)) &=  \int_{0}^{\floor{\sqrt{n}} } 2t^{1/2} \,dt + 
	O \left( \sum_{d=1}^{\floor{\sqrt{n}}}\frac{1}{\sqrt{d}}\right)
	+ O \left( \sqrt{n} \right)  \nonumber \\
	&=  \frac{4}{3}  n^{3/4} + O \left( \sqrt{n} \right),
\end{align}
after evaluating the integral. 
This verifies \eqref{eq:small-edges}.

Treating $\ADset[1,n]$, the elements in $\ADset[1,n] \setminus \ADsetsmall(n)$ are those of the form 
$d= \floor{n/k}$ for $1 \le k \le \sqrt{n}$
or $1 \leq k \leq \sqrt{n} - 1$. 
Each such $d$ has
$2 \sqrt{\floor{n/k}} +O(1)$ values $d' \in \ADset[1,n]$ having $d' \AD \floor{n/k}$.
These account for
all the remaining incidences in $\ADset[1,n]$, and their number $\Z^{\ast}$ satisfies
\begin{align*}
\Z^{\ast}   &= \sum_{k=1}^{\floor{\sqrt{n}}} \left( 2 \sqrt{\floorfrac{n}{k}} + O(1) \right)
=  \left(\sum_{k=1}^{\floor{\sqrt{n}}} 2\sqrt{\frac{n}{k}} \right)   + O(\sqrt{n} ),
\end{align*}
using $0\le \sqrt{x}- \sqrt{\floor{x}} \le 1$ for $x >0$. Approximating the sum by an integral,
we have
\begin{align*} 
 2 \sqrt{n} \left(\sum_{k=1}^{\floor{\sqrt{n}}} \frac{1}{\sqrt{k}} \right) = 2 \sqrt{n} \int_{t=0}^{\floor{\sqrt{n}}} \frac{dt}{\sqrt{t}} + O \left( \sqrt{n} \right),
\end{align*} 
where the  remainder term estimates the difference of   the $k$-th term of the sum with  the integral 
from $t=k-1$ and $t=k$   by $O ( k^{-3/2})$.
Evaluating the integral yields
\[
	\Z^{\ast} = 4 n^{3/4} + O \left(\sqrt{n} \right).
\]
Combining this bound with 
\eqref{eq:small-bound} verifies \eqref{eq:total-edges} for $\Z(\ADset[1,n])$, proving (a).

Treating $\ADsetlarge(n)$, according to  Theorem \ref{thm:upper-interval-struct} 
the order structure 
on $\ADsetlarge(n)$ under the map $\floor{n/k} \to k$ is
anti-isomorphic to the divisor order on the additive interval $k \in \{1, \ldots, \floor{\sqrt{n}}\}$. 
The total number of its incidences therefore matches  total number of incidences of the divisor order on 
the set $\{1, \ldots, \floor{\sqrt{n}}\}$.
Therefore 
\begin{equation} \label{eq:large-bound}
	\Z( \ADsetlarge (n)) = \sum_{k=1}^{\floor{\sqrt{n}}} \sigma_0(k), 
\end{equation} 
where $\sigma_0(k)$ counts the number of divisors of $k$. 
It was shown by
Dirichlet that
$$
\sum_{k=1}^{x} \, \sigma_0(k) = x \log x + (2 \gamma-1) x + O(x^{1/2}),
$$
where $\gamma \approx 0.57721$ is Euler's constant.
Substituting this bound in \eqref{eq:large-bound} 
establishes  the bound \eqref{eq:large-edge} for $\Z(\ADsetlarge(n))$,
completing the proof of (b).  

Treating $\Z( \ADsetsmall (n), \ADsetlarge(n))$, the
 bound in (c) is obtained using the  inclusion-exclusion formula
\[
	\Z( \ADsetsmall (n), \ADsetlarge(n) )= \Z(\ADset[1,n]) - \Z(\ADsetlarge(n)) - \Z(\ADsetsmall(n)) + O\left(n^{1/4}\right),
\]
with the remainder  term 
coming from the possible overlap of $\ADsetlarge(n)$ and $\ADsetsmall(n)$ in
the element $s= \floor{\sqrt{n}}$, which  has  $O( n^{1/4})$ lower incidences. 
The bounds in (a) and (b) yield the first estimate in (c). 
For the final estimate, there is no nonzero
contribution $\zeta(x,y)$ to $\Z( \ADsetlarge(n), \ADsetsmall(n) )$ except $\zeta(s,s)$ if $s$ belongs to
both $\ADsetlarge(n)$ and $\ADsetsmall(n)$.
\end{proof}

%
%
\subsection{Counting chains in intervals}
\label{subsec:count-chains}

We bound  the total  number of chains inside intervals of the floor quotient partial order,
treating the general interval $\ADset[d,n]$. 
The estimate has a sharp exponent  as $n \to \infty$, holding $d$ fixed.

We remark that the constant $\alpha_0 \approx 1.729$ in Theorem~\ref{thm:chain-upper-bound} 
appeared earlier in the work of Kalm\'{a}r~\cite{Kalmar:1931} and Klazar--Luca~\cite{KlazarL:07},
who studied the asymptotic growth of the number of chains in divisor order intervals $\sD[1,n]$, as $n \to \infty$.
%
%
\begin{defi} \label{defi:413}
A {\em  chain} $C$ between two elements $x,y$ of a poset $\sP$
with strict order $\prec$ is a set of elements $(a_0, a_1, \ldots, a_k)$ such that 
$x = a_0 \prec a_1 \prec \cdots \prec a_{k-1}\prec a_k= y$. 
Its {\em length} $\ell(C)= k$.
The set of all chains between $x,y$ is denoted $\sC_{\sP}(x,y)$,  or $\sC(x,y)$ if the
partial order is clear from context.
\end{defi} 

For any floor quotient interval $\ADset[d,n]$ we let $\TC(d,n)$ denote the total number of 
chains between $d$ and $n$.
In particular, $\TC(d,n)=0$ if $d \notAD n$, and 
 $\TC(n,n)=1$, as it counts a single chain of length $0$.
(We follow the conventions of Stanley \cite[Chap. 3]{Stanley:12}.)
%
%
\begin{thm}
\label{thm:chain-upper-bound}  
Let $\alpha_0 \approx 1.729$ denote the unique positive real value having 
$\zeta(\alpha_0) = \sum_{n\geq 1} n^{-\alpha_0} = 2$.
Then for all positive integers $d$ and $n$,
\begin{equation}\label{eq:chain-bound} 
\TC(d,n) \le \left( \frac{n}{d} \right)^{\alpha_0}.
\end{equation} 
\end{thm}

\begin{proof}
If $d$ is not a floor quotient of $n$, the left side on \eqref{eq:chain-bound} is zero, and the bound holds.
If $d = n$, then both sides of \eqref{eq:chain-bound} are equal to one.

Now suppose $d \AD n$ and $d\neq n$. 
We  write a chain $d = a_0 \AD a_1 \AD \cdots \AD a_{j-1} \AD a_j= n$, with $a_i \ne a_{i+1}$ for all $i$.
We count  the chains grouped  according to the value of $a_{j-1}$, which is strictly smaller than $n$.
Since $a_{j-1} \AD n$ and $a_{j-1} \neq n$, we have $a_{j-1}= \floorfrac{n}{k}$ for some integer $k \ge 2$.
There are $\TC(d, \floorfrac{n}{k})$
such chains with $a_{j-1} = \floorfrac{n}{k}$, 
so we obtain the upper bound
\begin{equation}\label{eq:upper-1}
\TC(d, n) \le \sum_{k=2}^n \TC\left(d, \floorfrac{n}{k}\right).
\end{equation}

We prove the result \eqref{eq:chain-bound} by fixing $d$ and applying induction on $n$. 
We have $\TC(d,d) = 1$ and $\TC(d, n) =0$ for $n < d$ and $d < n < 2d$, so the bound \eqref{eq:chain-bound}
holds for all $n < 2d$.
These are the base cases.
For the induction step, consider $n\geq 2d$ and assume
the induction  hypothesis holds for all $n' < n$.
We have
\begin{align*}
	\TC(d,n) &\le \sum_{k=2}^n \TC(d, \floorfrac{n}{k})\\
	&\leq \sum_{k=2}^n \left(\frac{1}{d} \floorfrac{n}{k} \right)^{\alpha_0} \\
	&\leq \sum_{k=2}^n \left( \frac{n}{kd} \right)^{\alpha_0} \\
	&\leq \left( \frac{n}{d} \right)^{\alpha_0} \sum_{k=2}^\infty \left( \frac{1}{k} \right)^{\alpha_0} 
	= \left( \frac{n}{d} \right)^{\alpha_0},
\end{align*}
where the last equality follows from our assumption that $\alpha_0$ satisfies $\displaystyle \sum_{k=1}^\infty \left( \frac{1}{k} \right)^{\alpha_0} = 2$.
This completes the induction step.
\end{proof} 

\begin{rmk}\label{rmk:312}
For fixed $d$ as $n \to \infty$, the exponent $\alpha_0$ in the upper bound \eqref{eq:chain-bound}
in Theorem \ref{thm:chain-upper-bound}
is sharp.  
If $d \AD n$ we have the lower bound
\begin{equation}\label{eq:lower-1}
\TC(d, n) \ge \sum_{k=2}^{\floor{\sqrt{n}}}  \TC(d, \floor{\frac{n}{k}}),
\end{equation}
since all chains counted on the right side are distinct.
Note that $d \AD n$ holds for all $n\geq d^2$, by Lemma \ref{lem:multiples-elem},
and the left side of \eqref{eq:lower-1} is then positive for $n \geq d^2$.
 One can  show by a straightforward induction argument,
that for fixed $d$ and any fixed small $\epsilon>0$, there is a constant $c_{d, \epsilon}>0$
depending on $d$ and $\epsilon$, such that
\[
	\TC(d,n) \ge c_{d, \epsilon} \left( \frac{n}{d} \right)^{\alpha_0 - \epsilon}
\]
holds for all $n \ge d^2$.
\end{rmk}

%
%
\section{Size of floor quotient intervals}
\label{sec:interval-size} 

In this section we prove bounds on the size of intervals in the floor quotient order.
We define the {\em  width}  as a measure of  the extent  of a floor quotient interval
  of a partial order on $\NNplus$ and compare it to 
  its cardinality.

\subsection{Width and size of floor quotient intervals } \label{subsec:51a}

\begin{defi} \label{def:width} 
 For a partial order $\sP$  on $\NNplus$ the {\em (multiplicative) width} $w(I) $ of a nonempty interval $I=\sP[m, n]$ is 
\begin{equation}
 w(\sP[m, n])  \coloneqq \frac{n}{m}.
\end{equation} 
It is a rational number, with  $w(\sP[m, n])  \ge 1$ for approximate divisor orders.
\end{defi}

The width $w(d,n)$  is invariant under scaling transformations: $(d, n) \to (ad, an)$.
Intervals of the floor quotient order $\sQ  = (\NNplus, \AD)$ are not scaling-invariant.
The set of intervals $\ADset [a, ak]$ having a fixed
integer multiplicative width  $w=k$ have sizes 
that vary with the scale parameter  $a \ge 1$.

\begin{exa}\label{exa:47}
The floor quotient intervals $\ADset [a,10a]$ for various $a$ are shown in the table below.
\begin{center}
\begin{tabular}{ccc}
 interval & elements & size \\ \hline \\[-0.8em]
$\ADset[1,10]$ & $\{ 1, 2, 3, 5, 10\}$ & 5 \\
$\ADset[2,20]$ & $\{ 2, 4, 5, 6, 10, 20\}$ & 6 \\
$\ADset[3,30]$ & $\{ 3, 6, 7, 10, 15, 30\}$ & 6 \\
$\ADset[4,40]$ & $\{ 4, 8, 13, 20, 40\}$ & 5 \\
$\ADset[9,90]$ & $\{ 9, 18, 45, 90\}$ & 4 \\
\end{tabular}

\end{center}
\end{exa}

\subsection{Upper bounds}\label{subsec:52a}

We  establish a general upper bound
on the size of  floor quotient order intervals $\ADset[d, n]$
as a function of their width $w \coloneqq {n}/{d}.$


\begin{thm}\label{thm:interval-bound} 
The floor quotient interval $\ADset[d,n]$ has size bounded by
\begin{equation}  \label{eqn:size-upper-bound} 
| \ADset[d,n]| \leq \frac{3}{2} \left( \frac{n}{d} \right)^{2/3}   .
\end{equation} 
In terms of  width,  all  floor quotient intervals $I$ satisfy 
the  bound
$
|I| \leq \frac{3}{2}w(I)^{2/3} .
$
\end{thm}

\begin{proof}
Recall that $\ADset[d,n]$ denotes the floor quotient interval
\[ 
\ADset[d,n] = \{ e \in \NNplus : d \AD e \AD n\}
\]
To obtain an upper bound for the size of $\ADset[d,n]$,
we use the set  inclusion
\begin{equation*} 
  \ADset[d,n] 
  \subset \{ e : d \AD e,\, d \leq e < \beta d \}  
  \bigcup \{ e : e \AD n,\, \beta d \leq e \leq n \},
\end{equation*}
 where $\beta \geq 1$ is an integer parameter to be suitably chosen later. 
 We obtain
\begin{equation}\label{eq:size-inclusion}
  |\ADset[d,n] | 
  \le  | \{ e : d \AD e,\, d \leq e < \beta d \} | 
  +  | \{ e :  e \AD n,\, \beta d \leq e \leq n \}|. 
\end{equation}
  We estimate the two terms on the right side of this bound  separately.
 \begin{enumerate}
 \item[(1)]  
In the ``small'' range,  $d \le e \le \beta d$, 
we use the bound from Lemma~\ref{lem:multiples-elem} that 
$d$ has at most $k$ floor multiples  in each subinterval  
$\{ kd \leq e < (k+1)d\}$.
We obtain 
\begin{align}\label{eqn:419} 
|\{e :  d\AD e,\, d \leq e < \beta d \}| 
&\leq 1 + 2 + \cdots + ({\beta} -1) 
= \frac{1}{2} {\beta}^2 - \frac12 {\beta}.
\end{align} 
(The second inequality is an equality if $\beta \leq d$.)
\item[(2)]
In the ``large'' range, $\beta d \le e \le n$
we have the equality of sets 
\[ 
 \{ e : e \AD n,\, \beta d \leq e \leq n \} 
 =  \{ \floor{\frac{n}{j}} : 1\leq j \leq \frac{n}{\beta d}\} 
 \]
(not counting  multiplicities on the right-hand side). 
 In consequence
 \begin{equation}\label{eqn:420} 
| \{ e : e \AD n,\, \beta d \leq e \leq n \} |\leq \floor{\frac{n}{\beta d}} 
\leq \frac1{\beta} \left( \frac{n}{d} \right).
\end{equation} 
\end{enumerate} 

The upper  bounds \eqref{eqn:419} and \eqref{eqn:420} 
 applied  to  \eqref{eq:size-inclusion} yield,   for each  
 $\beta \ge 1$, 
 \begin{equation} 
 \label{eq:ndiv-bound}
   |\ADset[d,n]| \leq  \frac{1}{2} {\beta}^2 - \frac12 {\beta} + \frac{1}{{\beta}} \left( \frac{n}{d} \right).
  \end{equation}
We now consider optimizing the choice of $\beta$.
 It suffices to show that for any real $x > 0$, we have
 \begin{equation}\label{eq:421}
 \min_{\beta \in \NNplus} \{ \frac12 \beta^2 - \frac{1}{2} \beta + \frac1\beta x \} 
 = \min\{ x, 1+\frac{1}{2 }x, 3 + \frac13 x, \ldots \}
 \leq \frac{3}{2} x^{2/3}.
 \end{equation}
 The expression
 $f(x) = \min_{\beta \in \NNplus} \{ \frac{1}{2} \beta^2 - \frac{1}{2 }\beta + \frac1\beta x \} $ 
 is a piecewise linear function of $x$,
 and its graph is the upper convex hull of the points
 $$
 \{ (0,0), (2,2), (12,7),\, (36,15),\, (80,26),\, \ldots \}
 = \{\textstyle (a^3+a^2 , \frac32 a^2 + \frac12 a) : a =0,1,2,\ldots \}.
 $$
 Each corner-point $(a^3+a^2, \frac{3}{2} a^2 + \frac{1}{2} a)$ of $f(x)$ lies below the graph of $g(x) = \frac{3}{2} x^{2/3}$.
 Since $g(x)$ is a concave function, this implies that the entire graph of $f(x)$ also lies below $g(x)$.
 This verifies the inequality \eqref{eq:421}.

 The desired result now follows from the bounds \eqref{eq:ndiv-bound} and \eqref{eq:421} with $x = \frac{n}{d}$.
 \end{proof}

The  upper bound of Theorem \ref{thm:interval-bound}  is asymptotically sharp for certain intervals 
$\ADset[d,n]$  having width $w = \frac{n}{d} \to \infty$.
We first consider an example, which we generalize in Proposition~\ref{prop:410}.

%
\begin{exa} \label{ex:410a} 
The  floor quotient interval $\ADset[10, 10000]$ corresponds to
 $d = 10$, $n=10000$, and width $w = 1000$.
It consists of the numbers
\begin{align*} 
\ADset[10, 10000] &= \{10\} \cup \{ 20, 21\} \cup\{30, 31, 32\} \cup \cdots \cup \{ 90, 91,\ldots , 98 \}\\
& \qquad  \bigcup  \{100= \floorfrac{10000}{100}, \floorfrac{10000}{99}, \ldots , \floor{\frac{10000}{2}}, \floorfrac{10000}{1}\} .
\end{align*} 
The size of this interval is 
\[ 
| \ADset[10,10000] |  = 1 + 2 + \cdots + 9 + 100 = 145 .\]
Theorem \ref{thm:interval-bound}  gives the upper bound
$
| \ADset[10, 10000]| \le \frac{3}{2} (1000)^{2/3} = 150 .
$

In  comparison, the size of the floor quotient interval
$\ADset[1,w] = \ADset[1,1000]$ is
\[ 
|\ADset[1,1000]|  =  2 \floor{ \sqrt{1000}} = 62.
\]
\end{exa}

%
%
\begin{prop}\label{prop:410} 
For any $d \geq 1$, 
the size of the  floor quotient order interval $\ADset[d,d^4]$ is
\[
|\ADset[d,d^4] | 
= \frac{3}{2}d^2 - \frac{1}{2}  d .
\]
In particular
$
|\ADset[d,n]| 
= \frac{3}{2} \left(\frac{n}{d}\right)^{2/3} - \frac12 \left(\frac{n}{d}\right)^{1/3}
$
when $n = d^4$.
\end{prop}

\begin{proof}
It suffices to  check that the bound 
\eqref{eq:size-inclusion} with $\beta = d$ 
is an equality in this case.
\end{proof} 

\subsection{Lower bounds}\label{subsec:53a} 

We obtain  lower bounds for the size of a  floor quotient interval  $\ADset[d,n]$ 
in terms of the parameters $n$ and $d$ separately. We allow empty intervals.

%
%
\begin{thm}\label{thm:interval-lower-bounds} 
 {\em (Floor quotient interval size lower  bounds)} 

\begin{enumerate}[(a)]
    \item
     If $1 \leq d \leq n^{1/4}$, then
    \[
    |\ADset[d,n] | \geq \frac{3}{2} n^{1/2} - \frac12 n^{1/4} - 1.\]
    \item
    If $n^{1/4} \leq d \leq n^{1/2}$, then 
    \[ 
    |\ADset[d,n]| \geq \frac{3}{2}\left( \frac{n}{d^2} \right) - \frac32 \left( \frac{n}{d^2} \right)^{1/2}.
    \]
    
    \item
    If $n^{1/2} \leq d \leq n$, then 
    \[
    |\ADset[d,n]| = \begin{cases}
    \sigma_0(k) & \text{if $d = \floor{ {n}/{k} }$
    for some } 1 \le k \le \sqrt{n} \\
    0 & \text{otherwise},
    \end{cases}
    \] 
    where $\sigma_0(k)= \sum_{d |k} 1 $ denotes the number of  divisors of $k$.
\end{enumerate}
\end{thm}

\begin{proof}
(a) The set of  floor multiples of $d$ contains all integers $m \geq d^2$,
and  the set of floor quotients  of $n$ contains all integers $m \leq \sqrt{n}$.
If $d \leq {n}^{1/4}$, then we have
\begin{equation*}
    \{ m : d \notAD m \}
    \subset \{ m : m \leq d^2\}
    \subset \{m : m\leq \sqrt{n} \}
    \subset \{ m  :  m \AD n \} .
\end{equation*}
This inclusion implies 
\begin{align*}
|\ADset[d,n]| =  
  | \{m : d \AD m \AD n \} |
 &=
 | \Big( \{ m  :  m \AD n\} \smallsetminus \{ m : d \notAD m \} \Big) |\\
 &\geq 
 |\{ m  :  m \AD n\}| - | \{ m : d \notAD m \} |.
\end{align*}
Using 
Corollary \ref{cor:interval-size}
 we have $ |\{ m  :  m \AD n\}| > 2 \sqrt{n} -2.$
By  Lemma~\ref{lem:multiples-elem} (b)    we have 
$| \{ m : d \notAD m \} | = \frac{1}{2} (d-1)(d+2).$
Substituting these estimates in the previous equation yields 
\begin{align*}
|\ADset[d,n]| 
 &> 
 \left( 2 \sqrt{n} -2\right) -  \frac{1}{2} (d-1)(d+2)  \\
 &= 2{n}^{1/2} - \frac{1}{2 }d^2 - \frac{1}{2} d - 1 .
\end{align*}
By the assumption that $d \leq n^{1/4}$,
this bound implies
$
|\ADset [d,n]| > \frac{3}{2} n^{1/2} - \frac{1}{2} n^{1/4} - 1.
$

(b) 
If ${n}^{1/4} < d \leq n^{1/2}$, then
\[
 \ADset[d,n] = \{m : d \AD m \AD n \} 
\supset \{m :  m \leq \sqrt{n},\, d \AD m \} 
\sqcup \{m :   d^2 \leq m ,\, m \AD n \} .
\]
The first set contains
\[ 
\{ m : m \leq \sqrt{n} ,\, d \AD m \} \supset \bigcup_{1 \leq k \leq \sqrt{n}/d-1} \{ kd, kd + 1, kd+2, \ldots, kd+k-1 \}
 \]
and the second set contains
\[
 \{ m : d^2 \leq m ,\, m \AD n \} = \{ \floor{\frac{n}{k}} : 1 \leq k \leq \frac{n}{d^2} \}.
\]
This implies
\begin{align*}
   | \ADset[d,n] | 
 &\geq |\{ m \leq \sqrt{n} : d \AD m \}|  + 
 |\{ d^2 \leq m : m \AD n \} |  \\
 &\geq \left(\frac{1}{2}  \floor{\frac{\sqrt{n}}{d}} ( \floor{\frac{\sqrt{n}}{d}} - 1) \right) + \left( \floor{\frac{n}{d^2}} \right) \\
 &\geq \left( \frac{1}{2} \left( \frac{\sqrt{n}}{d} - 2 \right)\left( \frac{\sqrt{n}}{d} - 1\right)\right) 
 + \left( \frac{n}{d^2}-1\right) \\
 &= \frac{3}{2}\left( \frac{n}{d^2} -  \frac{\sqrt{n}}{d}\right). 
\end{align*}

(c) If $d > \sqrt{n}$, 
then  Theorem~\ref{thm:upper-interval-struct} implies that 
the floor quotient interval $\ADset[d,n]$ is isomorphic to the
interval  $\sD[1, k]$ in the divisor order,
if $d = \floor{ \frac{n}{k}}$. 
Otherwise if $d \ne \floor{ \frac{n}{k}}$ for any $k$,
then $\ADset[d,n]$ is empty.
\end{proof}

\begin{rmk}\label{rmk:412} 
For intervals $\ADset[a, aw]$, holding the width $w$ fixed and varying $a \ge 1$,
 in qualitative terms it appears  that the size $|\ADset[a, aw]|$
 first increases and then decreases as $a \to \infty$.
\begin{enumerate}[(i)]
\item
 For $a=1$ the interval has size
$| \ADset[1, w] | = 2 {w}^{1/2} +O(1)$.  
\item
For $a = \floor{w^{1/3}}$,
the argument of Proposition \ref{prop:410} applies to give the size estimate
\[
| \ADset[a, aw] |
 = \frac{3}{2} w^{2/3} +O (w^{1/3}).
\]
 \item
For  $a \ge w$, $|\ADset[a,aw]| = \sigma_0(w) \ll w^{\epsilon}$
for any $\epsilon > 0$,
as $w \to \infty.$
\end{enumerate} 
For fixed $w$ we do not know whether the sequence of values $|\ADset[a, aw]|$ is unimodal as a function of $a \ge 1$. 
\end{rmk}

 %
 %
 \section{M\"{o}bius function of floor quotient order}
 \label{sec:mobius}  
 
 We review the classical M\"{o}bius function on the divisor partial order $(\NNplus,\divides)$ and then present results on
 the floor quotient M\"{o}bius function.    
 
The notion of the  two-variable M\"{o}bius function
of a locally finite partial order  was formalized  by   Rota \cite{Rota:64} in 1964, 
see also Zaslavsky~\cite{Zaslavsky:87} and Stanley~\cite[Chapter 3]{Stanley:12}.
M\"{o}bius functions for lattices and associated M\"{o}bius inversion formulae
have a history going back to the 1930's and earlier, 
see  Weisner \cite{Weisner:35}, Hall \cite{Hall:36}, and Ore \cite[pp. 181--207]{Ore:62}. 
Earlier work of Eric Temple Bell \cite{Bell:1915,Bell:1923} is relevant.

%
%
\subsection{Preliminaries on M\"{o}bius functions }\label{subsec:61}

 Given a partial order $\sP= (\NNplus, \preccurlyeq_\sP) $, 
 the M\"{o}bius function 
 of $\sP$ is the convolution inverse of the zeta function $\zeta_{\sP}(s,t)$
 of the partial order. 
That is, $\mu_\sP$ is the function in the incidence algebra 
 such that for all $d, n \in \NNplus$ we have
\[
  \sum_{{m \in \NNplus}} \mu_{\sP} (d, m)\zeta_{\sP}(m, n) = \begin{cases} 
  1  & \mbox{if} \,\, d= n \\ 
  0 & \mbox{if}\,\,  d \neq n
  \end{cases}. 
\]
 In other words, $\mu_\sP$ is the unique function $\mu_\sP : \NNplus \times \NNplus \to \ZZ$ satisfying
 the initial conditions 
  \begin{align}
  \mu_{\sP}(n,n) &=1 \qquad\text{for all }n \in \NNplus, 
  \\
 \mu_\sP(d,n) &= 0 \qquad\text{if } d \not\preccurlyeq_\sP n,  
  \end{align}
and the recursion  
\begin{equation}
\label{eqn:Mob1} 
 \sum_{m : \, d \,\preccurlyeq_{\sP}\, m \,\preccurlyeq_{\sP}\, n} \mu_{\sP} (d, m) =0
 \qquad \text{for all $d \neq n$ in }\NNplus.
\end{equation}

 The dual partial order $\sP^{\ast} = (\NNplus, \preccurlyeq_{\sP^\ast})$ to $\sP$ 
 has $d \preccurlyeq_{\sP} e \, \Leftrightarrow  \, e \preccurlyeq_{\sPdual} d$.
 The {\em dual M\"{o}bius function} satisfies  $\mu_{\sP^{\ast}} (n,d) = \mu_{\sP}(d,n).$
 In particular the relation \eqref{eqn:Mob1} for $\mu_{\sP^{\ast}}$ yields a dual recursion
 for $\mu_{\sP}$: 
 \begin{equation}\label{eqn:Mob2} 
 \sum_{ m:\, d \,\preccurlyeq_{\sP}\, m \,\preccurlyeq_{\sP}\, n} \mu_{\sP} (m,n) = 0
 \qquad \text{for all $d \neq n$ in }\NNplus.
 \end{equation}

For comparative purposes we recall the 
M\"{o}bius function of  the divisor order  $\sD = (\NNplus, \divides )$.
The classical
 M\"{o}bius function $\mu(n)$ is given by
 $$   
\mu(n) = \begin{cases} 
(-1)^k  & \,\mbox{if} \,\, n= p_1 p_2 \cdots p_k \,\, \mbox{is squarefree, where $p_i$ are prime} \\
0 & \,\mbox{if $n$ is not squarefree}.
\end{cases}
$$
For any $m, n$ the two-variable poset M\"{o}bius
function $\mu_\sD(m,n)$ reduces to a one-variable classical M\"{o}bius function: 
if $m \divides n$ then
$$
\mu_\sD(m, n) = \mu_\sD(1, \frac{n}{m})= \mu(\frac{n}{m}).
$$
This equality is a consequence of the scaling-invariance  property of divisor intervals.

%
%

\subsection{Floor quotient  M\"{o}bius function}\label{subsec:62}

The floor quotient M\"{o}bius function $\muAD(d, n)$ 
genuinely depends on both arguments. 
We show that for ``half'' of its inputs,
namely when
$d > \sqrt{n}$,
the value $\muAD(d, n)$ directly relates to the classical M\"{o}bius function.

We first  remark that many M\"{o}bius function
calculations can be made using  the basic two-variable M\"{o}bius function recursion
\begin{align}
\label{eqn:mu-recursion0}
 \muAD(d,n) &= - \sum_{\substack{d \AD e \AD n  \\ e \neq n}} \muAD(d,e)
 \qquad\text{when}\, \, n\neq d,
\end{align} 
 where the right side sums over the floor quotients $e \in \ADset[d,n] \setminus \{n\}$.

\begin{thm}
\label{thm:upper-interval-mobius} 
Fix a positive integer $n$. For any $k, \ell$ with $1 \leq \ell \le k  \leq \floor{\sqrt{n}}$,
the floor quotient M\"{o}bius function  satisfies
$$
\muAD (\floorfrac{n}{k}, \floorfrac{n}{\ell})  = \begin{cases} 
\mu(\frac{k}{\ell} ), & \mbox{if} \quad \ell \mid k\\
0 & \mbox{if} \quad \ell \nmid k,
\end{cases} 
$$
where $\mu(\cdot)$ is the classical M\"{o}bius function.
In particular, 
$$
\left| \muAD(\floor{\frac{n}{k}}, \floorfrac{n}{\ell}) \right| \le 1
\quad \mbox{if}  \quad 1 \le \ell \le k \leq \floor{ \sqrt{n}}.
$$
\end{thm}

\begin{proof}
The result follows from Theorem \ref{thm:upper-interval-struct}  and its proof. 
It states that
$\ADsetlarge(n)= \{ \floor{ \frac{n}{k} } \,: \, \, 1 \le k \le \floor{\sqrt{n}} \}$
is order anti-isomorphic to the divisor order
restricted to the initial additive segment $( \{1, \ldots, \floor{\sqrt{n}} \}, \divides~)$.
Specifically, for $1 \leq \ell \le k  \leq \floor{\sqrt{n}}$ we have
\begin{equation}\label{eq:reverse-divisor-order1} 
\floor{\frac{n}{k}} \AD \floor{\frac{n}{\ell}}
\qquad\text{if and only if}\qquad
\ell \divides k.
\end{equation} 
We have two cases: $\ell \mid k$ and $\ell \nmid k$.
\begin{enumerate}[(i)]
\item
$\ell \mid k$. 
In this case $\floor{\frac{n}{k}} \AD \floor{\frac{n}{\ell}}$.
 It is a general fact
for finite poset intervals that their M\"{o}bius values at interval endpoints agree  with that
of their duals. Thus
\[
\mu_1( \floorfrac{n}{k}, \floorfrac{n}{\ell}) = \mu_1^{\ast}( \floorfrac{n}{\ell}, \floorfrac{n}{k}),
\]
see for example Stanley \cite[Proposition 3.8.6 ff]{Stanley:12}. 
By the order-isomorphism 
\[
(\ADsetlarge(n))^{\ast}= (\ADsetlarge(n), {\,\preccurlyeq_{1}^{\ast}\,} )  \simeq  ( \sA[1, \floor{ \sqrt{n} }], \divides),
\]
we have 
\[
 \mu_1^{\ast}( \floorfrac{n}{\ell}, \floorfrac{n}{k})= \mu_{\sD}( \ell, k) = \mu(\frac{k}{\ell}).
\]
Combining these steps gives $\mu_1( \floorfrac{n}{k}, \floorfrac{n}{\ell}) = \mu({k}/{\ell})$.

\item  
$\ell \nmid k$. In this case $\floor{\frac{n}{k}} \notAD \floor{\frac{n}{\ell}}$.
By convention 
$
 \mu_1( \floorfrac{n}{k}, \floorfrac{n}{\ell}) = 0.
$ \qedhere
\end{enumerate} 
\end{proof} 


\begin{rmk}\label{rmk:72}
From these formulas, the floor quotient M\"{o}bius function values  $\muAD(d, n)$,
for $d < \sqrt{n}$, 
can in principle be recursively expressed as complicated (weighted) averages of classical M\"{o}bius function values.
Numerical results given in Section \ref{subsec:66} indicate the values of such $\muAD(d,n)$ can be very large. 
\end{rmk}

%
%
\begin{cor}\label{cor:63} 
For any $d\geq \sqrt{n}$,
the floor quotient M\"{o}bius function $\muAD$ satisfies
\begin{equation}
\muAD \left(d, n \right) = \begin{cases}
 \displaystyle \mu(\floorfrac{n}{d}) & \text{if } d \AD n ,\\
 0 &\text{otherwise},
 \end{cases}
\end{equation}
where $\mu(\cdot)$ is the classical M\"{o}bius function.
In particular, 
\begin{equation}
|\muAD (d, n)| \leq 1
\quad \mbox{if}  \quad d \geq \sqrt{n}.
\end{equation}
\end{cor}

\begin{proof}
If $d \geq \sqrt{n}$ and $d \AD n$,
 apply Theorem~\ref{thm:upper-interval-mobius} with $\ell=1$ and $k = \floorfrac{n}{d}$.
If $d\notAD n$, then $\muAD(d,n) =0$ by definition.
\end{proof}

%
%

\subsection{Upper bounds}\label{subsec:mobius-bounds}

We show that for the floor quotient poset, the  value of $\mu_1(d,n)$ is bounded above by a 
 function of polynomial growth in the width statistic $w(d,n)  = {n}/{d}$. 

A basic result of Philip Hall, treated by Rota \cite[Sec.3, Prop. 6]{Rota:64},
gives another way to compute M\"{o}bius function values. It 
is formulated in Stanley \cite[Prop. 3.8.5]{Stanley:12} as follows.

\begin{prop}[Philip Hall's theorem] 
\label{prop:PH} 
Let $\widehat P$ be a finite poset 
with smallest and largest elements
$\hat{0}$ and $\hat{1}$. 
Let $c_i$ be the number of chains
$\hat{0} = t_0 < t_1 < \cdots <t_{i-1} < t_i = \hat{1}$ of length $i$ between $\hat{0}$ and $\hat{1}$.
Then 
\begin{equation}\label{eq:PH} 
\mu_{\widehat P}(\hat{0}, \hat{1}) = c_0 - c_1+ c_2 - c_3 + \cdots.
\end{equation}
(If $\hat{0} \neq \hat{1}$ in $\widehat P$, then $c_0 = 0$ and $c_1 = 1$.) 
\end{prop} 
We may rewrite \eqref{eq:PH} as a sum over all chains,
\begin{equation}\label{eq:PH2}
\mu_{\widehat{P}}(\hat{0}, \hat{1})  = \sum_{ C \in \sC_{\widehat{P}}(\hat{0}, \hat{1}) } (-1)^{\ell(C)} .
\end{equation} 

We now establish an upper bound on M\"{o}bius function absolute values, valid for all $(d,n)$,
stated as Theorem \ref{thm:intro-mobius-bound}. We recall its statement for convenience.  

\begin{thm}
\label{thm:intro-mobius-bound2} 
Let $\muAD: \NNplus \times \NNplus \to \ZZ$ 
denote the M\"{o}bius function of the  floor quotient order. 
Let $\alpha_0 \approx 1.729$ denote the unique positive real given by $\zeta(\alpha_0)=2$,
where $\zeta(s)$ is the Riemann zeta function.
Then 
the following upper bound holds:
\begin{equation}
\label{eq:intro-mobius-bound2}
  |\muAD(d,n)| \leq  \left(\frac{n}{d}\right)^{\alpha_0} 
  \qquad \text{for all } d,n\in \NNplus .
\end{equation}
\end{thm}

\begin{proof}
We are to show
\[
  |\muAD(d, n)| \le \left( \frac{n}{d} \right)^{\alpha_0},
\]
where $\zeta(\alpha_0)=2$.
We may assume $d \AD n$, otherwise the bound trivially holds.

 Philip Hall's theorem implies the  absolute value bound
\begin{equation}
\label{eq:absolute-chain-bound} 
| \mu_{\widehat{P}}(\hat{0}, \hat{1})| \le c_0+c_1 + c_2 + \cdots = \TC_{\widehat{P}}( \hat{0}, \hat{1}),
\end{equation} 
for any poset interval $\widehat P$.
If $d \AD n$ we may choose the interval $\widehat{P} = \ADset[d, n]$ with $\hat{0}= d$ and $\hat{1}=n$ to obtain,
using \eqref{eq:absolute-chain-bound}, 
$$ | \muAD(d, n)| \le \TC(d, n).$$
 Theorem \ref{thm:chain-upper-bound} gives
$\TC(d,n) \le ({n}/{d})^{\alpha_0}$, completing the proof. 
\end{proof} 

\begin{rmk}
The bound of Theorem \ref{thm:intro-mobius-bound} makes no use of  sign cancellations of the terms
in the M\"{o}bius function. 
Numerical calculations in Section \ref{subsec:mobius-data} suggest polynomial growth of M\"{o}bius values 
for some value $n^{\beta_0}$. 
\end{rmk}

%
%

\subsection{Sign changes for the M\"{o}bius function}\label{subsec:64}

The following result addresses sign changes in the  M\"{o}bius values $\muAD(1,n)$  
as $n \geq 1$ varies, and shows they occur.

%
%
\begin{thm}
\label{thm:mobius-sign-change} 
There exists an infinite sequence of integers $\ell_1=2< \ell_2< \ell_3 < \cdots$ 
such that 
\begin{equation*}
(-1)^j \mu(1, \ell_j) > 0 \qquad\text{for each }j,
\end{equation*}
with $\ell_j$ satisfying the bounds 
\[
\ell_{j+1} \le 2\ell_j^2 - 2.
\]
Consequently the  floor quotient M\"{o}bius function $\muAD(n)= \muAD(1,n)$ has  infinitely many sign changes.
\end{thm} 

To prove this result we use recursions for the M\"{o}bius function 
given in the following two propositions.
Recall that if $d < n$, the basic recursion \eqref{eqn:mu-recursion0} gives
\begin{equation*}
\muAD(d,n) = - \sum_{\substack{\ell \in \ADset[d,n] \\ \ell < n}} \muAD(d, \ell).
\end{equation*}

%
%
\begin{prop}
\label{prop:mu-recursion0} 
If $d \AD n$ and $d < n$, then we have
\begin{equation}\label{eq:mu-cancel-recursion0}
\muAD(d,2n) = - \sum_{\substack{\ell \in \ADset[d,2n] \\ \ell \not\in \ADset[d, n] \\ \ell < 2n}} \muAD(d, \ell) .
\end{equation}
\end{prop} 

\begin{proof}
Since $d \AD n$ we have $\ADset[d, n] \subset \ADset[d,2n]$.
In the recursion for the M\"{o}bius value $\muAD(d,2n)$ we may separate the terms $\muAD(d,m)$ 
depending on whether or not $m \in \ADset[d,n]$: 
\begin{align*}
\muAD(d,2n) &= - \sum_{\substack{m < 2n \\ m \in \ADset[d,2n]}} \muAD(d,m) \\
&= - \sum_{\substack{m < 2n \\ m \in \ADset[d,2n] \smallsetminus \ADset[d,n]}}
 - \sum_{\substack{m \in \ADset[d,n]}} \muAD(d,m) \\
&= - \sum_{\substack{m < 2n \\ m \in \ADset[d,2n] \smallsetminus \ADset[d,n]} }
 \muAD(d,m) .
\end{align*}
In the middle line, the second summation sums to zero due to the defining M\"{o}bius recursion applied to the interval $\ADset[d,n]$;
this interval is nontrivial since $d \AD n$ and $d < n$.
\end{proof} 

We specialize Proposition \ref{prop:mu-recursion0} to initial intervals $\ADset[1,2n]$.
When $n = t^2-1$, the sum in \eqref{eq:mu-cancel-recursion0} then has the useful property
that all its nonzero terms have indices inside the interval $\ell \in [t, 2t^2 - 2]$.

%
%
\begin{prop}
\label{prop:mu-recursion1} 
 Let  $\muAD(n) \coloneqq\muAD(1,n)$.
 Then 
  the following hold.
\begin{enumerate}
\item[(1)]
For $n\geq 2$ we have
\begin{equation}
\label{eq:mu-cancel-recursion01}
\muAD(1, 2n) = - \sum_{\substack{\sqrt{n}<m < 2n \\ m \AD 2n \\ m \notAD n}} \muAD(1, m) .
\end{equation}

\item[(2)]
For $t \ge 2$ we have
\begin{equation}
\label{eq:mu-cancel-recursion1}
\muAD(1, 2t^2-2) = - \muAD(1, t) - \sum_{\substack{t < m < 2t^2-2 \\ m \AD 2t^2-2 \\  m \notAD t^2-1}}  \muAD(1, m).
\end{equation} 
\end{enumerate}
\end{prop} 

\begin{proof} 
(1) This is Proposition~\ref{prop:mu-recursion0} with $d = 1$, with the observation that $m \notAD n$ implies
 $m > \sqrt{n}$
by Proposition~\ref{prop:interval-complement} (2).

(2) Apply (1) with $n = t^2-1$.
We observe that 
$t \AD 2t^2-2$ and 
$t \notAD t^2-1$,
so the term $\muAD(t,t)$ appears in the right-hand sum of \eqref{eq:mu-cancel-recursion01}.
Moreover all $m$ in the rightmost sum satisfy $m \geq t$, because 
of the bound $m > \sqrt{t^2 - 1}$ restricting the summation.
\end{proof}

The following  example illustrates one case of the  recursion \eqref{eq:mu-cancel-recursion1}.
\begin{exa}\label{exa:65a} 
 We abbreviate $\muAD(n)\coloneqq \muAD(1,n)$. 
 For $t=5$ we have $2t^2 - 2 = 48$. 
 The M\"{o}bius recurrence \eqref{eqn:Mob1} becomes
 \[
 \muAD(48) = - \muAD(24) -\muAD(16) - \muAD(12)- \muAD(9) - \muAD(8) - \sum_{d=1}^{6} \muAD(d).
 \]
 For $t^2 - 1 = 24$, the recurrence becomes
 \[
 -\muAD(24) = \muAD(12)+ \muAD(8) + \muAD(6) + \sum_{d=1}^4 \muAD(d). 
 \]
 Substituting this formula in the one above yields 
  $$
 \muAD(48) = - \muAD(16) - \muAD(9) - \muAD(5)
 \qquad\text{or}\qquad
 \muAD(5)= - \muAD(48) - \muAD(16) - \muAD(9) .
 $$
\end{exa} 

%
%
\begin{proof}[Proof of Theorem~\ref{thm:mobius-sign-change}]
We can write equation~\eqref{eq:mu-cancel-recursion1} in Proposition \ref{prop:mu-recursion1} as
\begin{equation}
\label{eq:mu-cancel-recursion2}
\muAD(1, t)  = - \sum_{\substack{t < m \leq 2t^2-2 \\ m \AD 2t^2-2 \\ m \notAD t^2-1}}  \muAD(1, m),
\end{equation} 
an  essential point being that all terms on the right-hand side are $\muAD(1,\cdot)$-values on inputs 
between $t+1$ and $2t^2-2$. 

To prove the theorem we
proceed by induction on the index $j$ in $\ell_j$.
For the base case, let $\ell_1 = 2$ so that  $\muAD(1,\ell_1) = \muAD(1,2) = -1$.
For the induction step, suppose $\ell_j$ is given with $(-1)^j \muAD(1,\ell_j) > 0$; 
we must then find some $\ell_{j+1}$ satisfying
\[ 
	\ell_j < \ell_{j+1} \leq 2\ell_j^2 - 2
	\qquad\text{and}\qquad
	\muAD(1, \ell_j) \muAD(1, \ell_{j+1}) < 0.
\]
Take \eqref{eq:mu-cancel-recursion2} with $t = \ell_j$ and multiply both sides by $\muAD(1, \ell_j)$,
\[
\muAD(1, \ell_j)^2  = - \sum_{\substack{t < m \leq 2\ell_j^2-2 \\ m \AD 2\ell_j^2-2 \\ m \notAD \ell_j^2-1}}  \muAD(1,\ell_j) \muAD(1, m).
\]
The left-hand side is (strictly) positive by assumption that $\muAD(1,\ell_j) \neq 0$, so there must be some term in the right-hand side sum which is (strictly) negative. 
Let $\ell_{j+1}$ be the index $m$ of any such term.
This proves the induction hypothesis.
\end{proof}

\begin{rmk}
\label{rmk:constant-sign}
Numerical data given in Section \ref{subsec:mobius-data}
exhibits $\muAD(1,n)$  having a  constant sign over very long intervals.
Perhaps the length of these intervals can grow like 
  $n^c$ for some $c>0$.
\end{rmk}

%
%

\subsection{Differenced M\"{o}bius function}
\label{subsec:diff-mobius} 

 The intervals  $\ADset[1,n]$
 and $\ADset[1, n-1]$ in the floor quotient order have many common elements  and their difference
 set is small; this property suggests study of the differenced 
 floor quotient M\"{o}bius function, fixing  $d=1$.

  In what follows we use the abbreviation
 $$\mu_1(n) : = \mu_{1}(1, n) .$$ 
 We show that $\mu_1(n)$  satisfies linear recursions involving 
 its first forward-difference.

\begin{defi} \label{defi:62} 
The {\em differenced floor quotient M\"{o}bius function} for $n\geq 1$ is
\begin{equation} 
\Delta \mu_{1}(n) \coloneqq \muAD( n) - \muAD(n-1),
\end{equation}
where $\muAD(n) \coloneqq \muAD(1, n)$. 
We adopt the convention that $\muAD(0) \coloneqq 0$.
\end{defi}

The difference operator $\Delta$ gives the basic recursion
\begin{equation}\label{eq:mu-recursion} 
\muAD(n) =  \sum_{m=1}^n \Delta \mu_{1} (m). 
\end{equation}

The next result shows that $\Delta\muAD(n)$ satisfies an interesting recursion in terms of
a subset of the proper  divisors of $n$. 
At special values
$n= s^2$ and $n=s(s+1)$ 
the recursion involves undifferenced values $\muAD(s)$.

\begin{thm}\label{thm:mobius-recursion}
For $n\geq 3$, the differenced M\"{o}bius function satisfies the recursion 
\begin{equation} \label{eqn:diff-recursion0} 
\Delta \muAD(n) = \begin{cases}
\displaystyle
- \sum_{\substack{ d \divides n\\  \sqrt{n} < d < {n} }} \Delta \muAD(d) & \text{if }n \ne s^2\,\, \mbox{or}\,\, s(s+1),\\[3em]
\displaystyle
- \Bigg( \sum_{\substack{d \divides n \\ \sqrt{n} < d < n }} \Delta \muAD(d) \Bigg)  - \muAD(s) & \text{if }n = s^2 \text{ or } s(s+1) .
 \end{cases}
\end{equation}
Note that $s = \floor{\sqrt{n}}$ when $n=s^2$ or $s(s+1)$.
\end{thm}

\begin{proof}[Proof of Theorem~\ref{thm:mobius-recursion}] 
We split the recursion \eqref{eqn:mu-recursion0} into two parts as
 \begin{align} \label{eqn:mu-recursion}
 \muAD(n)  &=   - \sum_{\substack{d \in \ADsetlarge(n) \\ d \neq n}} \muAD(d)
  - \sum_{\substack{d \in \ADsetsmall(n) \\ d \not\in \ADsetlarge}} \muAD(d),
\end{align}
using the abbreviated notation $\muAD(n) \coloneqq \muAD(1,n)$.

Set $s = \floor{ \sqrt{n} }$, so $s^2 \le n < (s+1)^2$.
 Using Proposition~\ref{prop:interval-complement} the  equation \eqref{eqn:mu-recursion} becomes
 \begin{align} \label{eqn:mu-recursion1}
 \muAD(n)  &=  \begin{cases}
 \displaystyle
  - \sum_{\substack{2\le k \le s}} \muAD(\floor{ \frac{n}{k} })
  - \sum_{\substack{1\leq d \leq s-1}} \muAD(d)
  &\text{if } s^2 \leq n < s(s+1), \\[2em]
  \displaystyle
  - \sum_{\substack{2\le k \le s}} \muAD(\floor{ \frac{n}{k} })
  - \sum_{\substack{1\leq d \leq s}} \muAD(d)
  &\text{if } s(s+1) \leq n < (s+1)^2.
  \end{cases}
\end{align}

We will difference the  equations above  at $n$ and $n-1$. 
We have three cases:   the generic case where $n \ne s^2, s(s+1)$  
and  two exceptional cases $n = s^2$ and $n = s(s+1)$.
 Lemma~\ref{lem:consecutive-interval-sets} describes  the overlap of the intervals $\ADset[1,n]$ and $\ADset[1,n-1]$,
 as sets.

For the generic case 
we obtain,  from  \eqref{eqn:mu-recursion1},
that 
\begin{align}\label{eqn:mu-recursion-3a} 
\Delta \muAD(n) &= - \sum_{2 \le k \le s} \left( \muAD(\floorfrac{n}{k} )
 - \muAD(\floorfrac{n-1}{k} )\right) 
\end{align}
because  the second sums over $d$ in \eqref{eqn:mu-recursion1} cancel for $n$ and $n-1$.
To simplify the summand, we have
\[
\floorfrac{n}{k} - \floorfrac{n-1}{k} = \begin{cases}
1 &\text{if } k \divides n, \\
0 &\text{otherwise},
\end{cases}
\]
so \eqref{eqn:mu-recursion-3a} becomes
\begin{align}\label{eqn:mu-recursion-4}
\Delta \muAD(n) 
 &= - \sum_{\substack{k\divides n \\2\le k \le s}} \Delta \muAD(\floorfrac{n}{k} )
 = - \sum_{\substack{k\divides n \\2\le k \le s}} \Delta \muAD\left( \frac{n}{k} \right).
\end{align} 
The sum in \eqref{eqn:mu-recursion-4} 
agrees with  the right side of \eqref{eqn:diff-recursion0}, after substituting $d = n/k$.

We next treat the case $n= s^2$. 
In this case $\floor{ \sqrt{n-1}}= s-1$.
Comparing $\muAD(n)$ and $\muAD(n-1)$ in \eqref{eqn:mu-recursion1} the second sums over $d$ cancel, but the first sum over $k$
for $\muAD(n-1)$ has one fewer term, hence the term
$\muAD(\floor{ {n}/{s} }) = \muAD(s)$ for $n$ remains uncancelled. 
We obtain 
\begin{align*}
\Delta \muAD( n) 
&= - \Bigg( \sum_{2\le k \le s-1} \muAD(\floorfrac{n}{k} ) - \muAD(\floorfrac{n-1}{k} )\Bigg)  - \muAD(s) \\
&= - \Bigg( \sum_{\substack{k\divides n \\2\le k \le s-1}}\Delta  \muAD(\floorfrac{n}{k} ) \Bigg) - \muAD(s)
= - \Bigg( \sum_{\substack{d \divides n \\ \sqrt{n} < d < n}} \Delta\muAD(d) \Bigg) - \muAD(s).
\end{align*} 
The latter sum agrees with \eqref{eqn:diff-recursion0}
in the case $n=s^2$. 

Finally we treat the case  $n= s(s+1) $. 
In this case $\sqrt{n-1} = s$ 
but in \eqref{eqn:mu-recursion1} $\muAD(n)$ falls in the lower case while $\muAD(n-1)$ falls in the upper case.
Hence the second sum over $d$ for $\muAD(n-1)$ is one element shorter  than for $\muAD(n)$, 
and we again get
\begin{align}
\Delta \muAD(n) &= - \left( \sum_{2\le k \le s} \muAD(\floorfrac{n}{k} ) - \muAD(\floorfrac{n-1}{k} )\right)  - \muAD( s) \nonumber \\
&= - \Bigg( \sum_{\substack{k\divides n \\2\le k \le s}}\Delta  \muAD(\floorfrac{n}{k} )\Bigg) - \muAD(s)
= - \Bigg( \sum_{\substack{d \divides n \\ \sqrt{n} < d < n}} \Delta\muAD(d) \Bigg) - \muAD(s). \nonumber
\end{align} 
The result agrees with \eqref{eqn:diff-recursion0}
in the case $n=s(s+1)$.
\end{proof}

We present  examples of the recursion in Theorem~\ref{thm:mobius-recursion}. 

 \begin{exa}\label{exa:66} 
The case  $n=15$ is a generic example where  $n \ne s^2$and  $n \ne s(s+1)$. 
We have 
\begin{align*}
\Delta\mu_{1}(15) 
&= \mu_{1}(15) - \mu_{1}(14) \\
 &=  - \left(\mu_{1}(7) + \mu_{1}(5) + \mu_{1}(3) + \mu_{1}(2) + \mu_{1}(1)\right) \\
 & \quad + \left(\mu_{1}(7) + \mu_{1}(4) + \mu_{1}(3) + \mu_{1}(2) + \mu_{1}(1)\right) \\
 &= -\Delta\mu_{1}(5)  .
 \end{align*}
\end{exa} 

\begin{exa}\label{exa:67} 
In the special cases  $n= s^2$ and  $n=s(s+1)$ the recursions contain extra terms involving
 the  {\em undifferenced} floor quotient M\"{o}bius function
An example with $n=s(s+1)$ is  $n=12= 3\cdot 4$, where $s = \floor{ \sqrt{12} }= 3$.
\begin{align*}
\Delta\muAD(12) 
&= \muAD(12) - \muAD(11) \\
 &=  - \left(\muAD(6) + \muAD(4) + \muAD(3) + \muAD(2) + \muAD(1)\right) \\
 & \quad + \left(\muAD(5)  + \muAD(3) \qquad\qquad+ \muAD(2) + \muAD(1)\right) \\
 &= -\Delta\muAD(6) - \Delta\muAD(4) - \muAD(3) .
 \end{align*}
 An example with $n=s^2$ is  $n=16 = 4^2$, where $s=\floor{ \sqrt{n} } = 4$,
\begin{align*}
\Delta\muAD(16) 
&= \muAD(16) - \muAD(15) \\
 &=  - \left(\muAD(8) + \muAD(5) + \muAD(4) + \muAD(3) + \muAD(2) + \muAD(1)\right)  \\
 & \quad + \left(\muAD(7) + \muAD(5) \qquad\qquad + \muAD(3) + \muAD(2) + \muAD(1)\right) \\
 &= -\Delta\muAD(8)\qquad\quad - \muAD(4)  .
 \end{align*}
\end{exa}

%
%

\subsection{Support of differenced M\"{o}bius function}\label{subsec:66} 

The differenced floor quotient M\"{o}bius function vanishes for many $n$.
The following result gives several conditions on $n$ under which $\Delta \mu_1(n)=0$. 
Condition (b) implies  that $\Delta \mu_1(n)=0$ for a positive fraction of integers $n$; 
the set of odd squarefree integers is known to have natural density ${4}/{\pi^2} \approx 0.405$.


\begin{prop}\label{th:65} 
\hfill
\begin{enumerate}[(a)]
\item
If $p\geq 3$ is a prime number, then $\Delta\muAD(p) = 0$.
\item
If $n$ is odd and square-free, then $\Delta\muAD(n) = 0$.
\item
If $n$ has a prime divisor $p$ such that $p \geq \sqrt{n}+1$,
then $\Delta\muAD(n) = 0$.
\end{enumerate}
\end{prop}
\begin{proof}
(a) We use the recursion \eqref{eqn:diff-recursion0} with $n=p$.
 We have $p \ne s^2$ or $ s(s+1)$,
so  using the first equation, the sum over divisors $2\leq d < \sqrt{p}$ is empty, 
so $\Delta\muAD(1,p) = 0$.

(b) 
We apply induction on the number of prime factors of 
$n$.
The base case is established in (a). 
We use the recursion \eqref{eqn:diff-recursion0}.  
Since $n$ is odd and squarefree, we have $n \ne s^2$ or $s(s+1)$. 
In this case all terms in the sum over  $2\leq d < \sqrt{n}$ are 
$\Delta \mu_{1}(\cdot)$ evaluated at  odd squarefree 
numbers having a fewer number of prime factors, 
so all terms are zero by the induction hypothesis.
This completes the induction step.

(c)  The result follows by induction on the number of prime factors of $n$. 
The base case is established by (a), since $p \ge 3$. 
We use the recursion \eqref{eqn:diff-recursion0}. 
 If $p\ge  \sqrt{n}+1$ divides $n$, then $n \ne s^2$ or  $s(s+1)$.
Now each of the  terms $n/d$ in the sum over $2\leq d < \sqrt{n}$ separately  is divisible by $p$,
 has $p> \sqrt{n}{d} +1$, and has fewer prime factors than $n$. By the induction hypothesis
 all such terms are $0$, completing the induction step.
\end{proof}

%
%

\subsection{Numerical computation of $\mu_{1}(1,n)$}\label{subsec:mobius-data}

The results above give two methods to compute $\muAD(1,n)$.
\begin{enumerate}
\item
The values  $\muAD(1,n)$ may be computed recursively from values $\muAD(1, k)$ for
smaller $k$ using the M\"{o}bius recursion \eqref{eqn:mu-recursion0}.
\item
Alternatively one may  compute $\muAD(1,n)$ 
using the recursion \eqref{eqn:diff-recursion0} 
 for the differenced M\"{o}bius function
combined with the  summation recursion \eqref{eq:mu-recursion}. 
\end{enumerate}
 
The sums in method (2) using the differenced recursion in Theorem \ref{thm:mobius-recursion}  have
many fewer terms than those in method (1).

We present computational results below. 
Figure~\ref{fig:mobius-10e4} presents data on $\muAD(1,n)$ for $n$ up to $10,000$;
the graph seems to stay within the bound $O(n^{1/2})$.

\begin{figure}[h]
\includegraphics[scale=0.7]{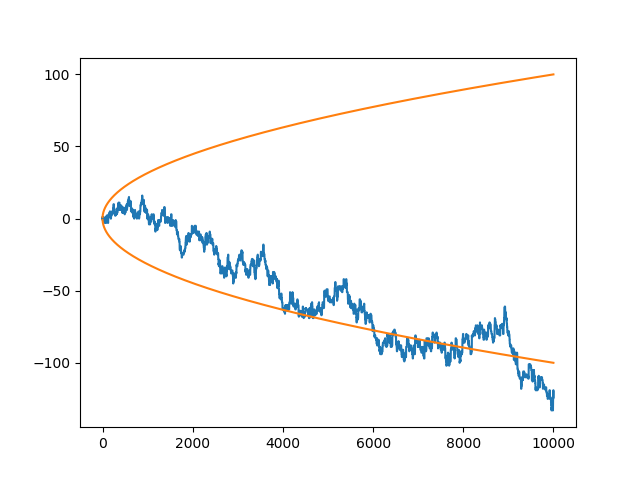}
\caption{M\"{o}bius function $\muAD(1,n)$ up to $n=10,\!000$.
The smooth curve shows $y= \pm \sqrt{x}$. }
\label{fig:mobius-10e4}
\end{figure} 

On a larger scale, Figure \ref{fig:mobius-10e6} graphs  $\muAD(1,n)$ for $n$ up to $1,000,000$.
Compared to Figure \ref{fig:mobius-10e4} the values 
reverse  direction of growth and also grows far outside the square-root bound.
Empirically one estimates $|\muAD(1,n)| = \Omega( n^{0.63})$ on this range. 


\begin{figure}[h]
\includegraphics[scale=0.7]{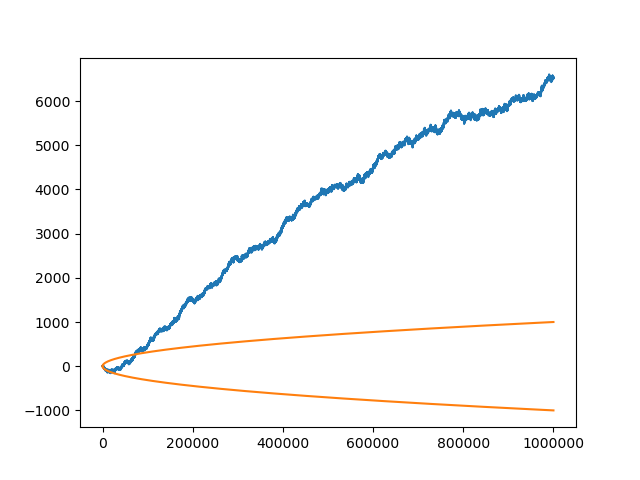}
\caption{Floor quotient M\"{o}bius function $\muAD(1,n)$ up to $n=1,\!000,\!000$. The smooth curve shows
$y=\pm {x}^{1/2}$.}
\label{fig:mobius-10e6}
\end{figure}

%
%
\section{Concluding remarks}\label{sec:concluding} 

We raise four 
directions for further investigation.

\subsection{Behavior of  floor quotient M\"{o}bius function} \label{subsec:81}

 Very  simple questions about the floor quotient M\"{o}bius function remain  unanswered.

(1) Does the function $h(n) \coloneqq \muAD(1,n)$ grow arbitrarily large?
This is suggested by the empirical data.

(2) What are the best asymptotic  growth bounds on the  M\"{o}bius function?
 It seems possible  that 
 $|\muAD(1,n) | \ll n^{1+\epsilon}$ may hold.
 
 The differenced M\"{o}bius function recurrences presented in Section~\ref{subsec:diff-mobius} for initial
 floor quotient intervals give another approach  which may
 lead to improved upper bounds on the growth of $|\muAD(1,n) |$.
  Analysis of this  recurrence 
 may be a challenging problem.


\subsection{Incidence algebras of floor quotient intervals}
\label{subsec:82}

The zeta function and M\"{o}bius function of a partial order are special members of the incidence algebra attached
to the partial order.  
The incidence algebra  of a poset $(\NNplus, \preccurlyeq)$ is the set of $\ZZ$-valued functions on the set of pairs 
$(d,n) \text{ such that } d \preccurlyeq_1 n$,
 see Smith \cite{Smith:67,Smith:69a,Smith:69b}. 
 For general structures of incidence algebras over a ring $R$, 
 which are in general non-commutative associative algebras, 
 see Stanley \cite{Stanley:70} and Spiegel and O'Donnell \cite{SpiegelO:97}.

The floor quotient incidence algebra  $\sI_{n}(\ZZ)= \sI(\ADset[1,n])$ of an initial interval 
 can be realized as a 
(generally noncommutative) ring of integer matrices with rows and columns of size $|\ADset[1,n]| \approx 2\floor{\sqrt{n}}$.
Allowing $\QQ$-coefficients Theorem \ref{thm:intro-interval-incidence} 
implies that the  incidence algebra of $\ADset[1, n]$
has $\dim_{\QQ}( \sI_n (\QQ)) \sim \frac{16}{3} n^{3/4}$, as $n \to \infty$.

Cardinal's algebra  $\mathbf{\sA}_n$, regarded as a $\ZZ$-algebra,  can 
 be identified with  a commutative  $\ZZ$-subalgebra
of the incidence algebra of the initial interval $\ADset[1, n]$ of the floor quotient poset.
Further study of the relations of these two $\ZZ$-algebras may shed light on relations between the zeta function and M\"{o}bius function
of the  floor quotient poset and that of the Cardinal algebra $\mathbf{\sA}_n$. 

\subsection{Generalizations of floor quotient orders: $a$-floor quotient orders} \label{subsec:83}

  One can generalize the floor quotient  partial order into a family of partial orders, 
   whose existence as a non-trivial family is a consequence of the failure of scaling-invariance
  for the floor quotient order.
  Define for each $a \ge 1$
  the {\em $a$-floor quotient relation}  $\ADa$  
  by $d \ADa n$ if and only if $ad \AD an$.  
   One can show these relations define a family of approximate divisor orders
 $\sQ_a = (\NNplus, \ADa)$ for  $a \ge 1$.  
 The case $a=1$ is the floor quotient relation.
 We will  treat their properties in \cite{LagR:22b}. 
 
   A remarkable property of the family of $a$-floor quotient relations is that as $a \to \infty$
  these partial orders converge monotonically  
  to the divisor order $\sD= (\NNplus, \divides)$.
  That is, some ordered pairs $(d, n)$ are removed from the relation $d \ADa n$ as $a$ increases, but none are ever added. 
  Thus these partial orders interpolate between the floor quotient order and the divisor order. 
 The local rate of convergence can be analyzed. When restricted to intervals $\ADset_a[d, n]$,
 one can study the convergence of the $a$-floor quotient M\"{o}bius function $\muADa$ as $a$ varies
 to the divisor order M\"{o}bius function. The convergence is not monotone.

\subsection{Approximate divisor orderings interpolating towards the additive total order} 
\label{subsec:84}

The floor quotient order is in some sense ``halfway between'' the multiplicative divisor order and the additive total order on $\NNplus$. 
The $a$-floor quotient orders interpolate between the floor quotient order and the divisor order.  
Are there any interesting families of approximate divisor orders on $\NNplus$ that interpolate in the opposite direction, 
between the floor quotient order and the additive total order?

\bibliographystyle{amsplain}

\bibliography{floor-quotients-ref}

\providecommand{\bysame}{\leavevmode\hbox to3em{\hrulefill}\thinspace}
\providecommand{\MR}{\relax\ifhmode\unskip\space\fi MR }
\providecommand{\MRhref}[2]{%
  \href{http://www.ams.org/mathscinet-getitem?mr=#1}{#2}
}
\providecommand{\href}[2]{#2}
\begin{thebibliography}{10}

\bibitem{AssiG:16}
A.~Assi and P.A. Garc\'{\i}a-S\'{a}nchez, \emph{Numerical semigroups and
  applications}, RSME Springer Series, vol.~1, Springer, [Cham], 2016.
  \MR{3558713}

\bibitem{Bell:1915}
E.T. Bell, \emph{An arithmetical theory of certain numerical functions},
  University of Washington, Seattle, 1915.

\bibitem{Bell:1923}
\bysame, \emph{Euler algebra}, Trans. Amer. Math. Soc. \textbf{25} (1923),
  no.~1, 135--154. \MR{1501234}

\bibitem{Cardinal:10}
J.-P. Cardinal, \emph{Symmetric matrices related to the {M}ertens function},
  Linear Algebra Appl. \textbf{432} (2010), no.~1, 161--172. \MR{2566467}

\bibitem{CardinalO:20}
J.-P. Cardinal and M.~Overholt, \emph{A variant of {M}\"{o}bius inversion},
  Exp. Math. \textbf{29} (2020), no.~3, 247--252. \MR{4134825}

\bibitem{Dirichlet:1849}
G.L. Dirichlet, \emph{\"{U}ber die {B}estimmung der {M}ittleren {W}ert in der
  {Z}ahlentheories}, Abhandl. Akad. Wiss. (1849), 69--83.

\bibitem{Hall:36}
P.~Hall, \emph{The {E}ulerian functions of a group}, Q. J. Math. \textbf{1}
  (1936), 134--151.

\bibitem{Heyman:19}
R.~Heyman, \emph{Cardinality of a floor function set}, Integers \textbf{19}
  (2019), Paper No. A67, 7. \MR{4188747}

\bibitem{Kalmar:1931}
L.~Kalm\'{a}r, \emph{A ``factorisatio numerorum'' probl\'{e}m\'{a}j\'{a}rol},
  Mat. Fiz. Lapok \textbf{38} (1931), 1--15.

\bibitem{KlazarL:07}
M.~Klazar and F.~Luca, \emph{On the maximal order of numbers in the
  ``factorisatio numerorum'' problem}, J. Number Theory \textbf{124} (2007),
  no.~2, 470--490. \MR{2321375}

\bibitem{Koren:02}
I.~Koren, \emph{Computer arithmetic algorithms}, Prentice Hall, Inc., Englewood
  Cliffs, NJ, 1993. \MR{1206231}

\bibitem{LMR:16}
J.C. Lagarias, T.~Murayama, and D.H. Richman, \emph{Dilated floor functions
  that commute}, Amer. Math. Monthly \textbf{123} (2016), no.~10, 1033--1038.
  \MR{3593644}

\bibitem{LagR:22b}
J.C. Lagarias and D.H. Richman, \emph{The family of $a$-floor quotient partial
  orders}, in preparation.

\bibitem{Ore:62}
O.~Ore, \emph{Theory of graphs}, American Mathematical Society Colloquium
  Publications, Vol. XXXVIII, American Mathematical Society, Providence, R.I.,
  1962. \MR{0150753}

\bibitem{RamirezA:05}
J.L. Ram\'{\i}rez~Alfons\'{\i}n, \emph{The {D}iophantine {F}robenius problem},
  Oxford Lecture Series in Mathematics and its Applications, vol.~30, Oxford
  University Press, Oxford, 2005. \MR{2260521}

\bibitem{Rota:64}
G.-C. Rota, \emph{On the foundations of combinatorial theory. {I}. {T}heory of
  {M}\"{o}bius functions}, Z. Wahrscheinlichkeitstheorie und Verw. Gebiete
  \textbf{2} (1964), 340--368 (1964). \MR{174487}

\bibitem{Smith:67}
D.A. Smith, \emph{Incidence functions as generalized arithmetic functions.
  {I}}, Duke Math. J. \textbf{34} (1967), 617--633. \MR{219465}

\bibitem{Smith:69a}
\bysame, \emph{Incidence functions as generalized arithmetic functions. {II}},
  Duke Math. J. \textbf{36} (1969), 15--30. \MR{242757}

\bibitem{Smith:69b}
\bysame, \emph{Incidence functions as generalized arithmetic functions. {III}},
  Duke Math. J. \textbf{36} (1969), 353--367. \MR{242758}

\bibitem{SpiegelO:97}
E.~Spiegel and C.J. O'Donnell, \emph{Incidence algebras}, Monographs and
  Textbooks in Pure and Applied Mathematics, vol. 206, Marcel Dekker, Inc., New
  York, 1997. \MR{1445562}

\bibitem{Stanley:70}
R.P. Stanley, \emph{Structure of incidence algebras and their automorphism
  groups}, Bull. Amer. Math. Soc. \textbf{76} (1970), 1236--1239. \MR{263718}

\bibitem{Stanley:12}
\bysame, \emph{Enumerative combinatorics. {V}olume 1}, second ed., Cambridge
  Studies in Advanced Mathematics, vol.~49, Cambridge University Press,
  Cambridge, 2012. \MR{2868112}

\bibitem{Weisner:35}
L.~Weisner, \emph{Abstract theory of inversion of finite series}, Trans. Amer.
  Math. Soc. \textbf{38} (1935), no.~3, 474--484. \MR{1501822}

\bibitem{Zaslavsky:87}
T.~Zaslavsky, \emph{The {M}\"{o}bius function and the characteristic
  polynomial}, Combinatorial geometries, Encyclopedia Math. Appl., vol.~29,
  Cambridge Univ. Press, Cambridge, 1987, pp.~114--138. \MR{921071}

\end{thebibliography}

\end{document}